\documentclass{amsart}
\usepackage{amsmath, amssymb}
\usepackage{amsfonts}
\usepackage{amsthm, upref}
\usepackage{graphicx}
\usepackage{enumerate}
\usepackage[usenames, dvipsnames]{color}
\usepackage{mathtools}

\usepackage{ifthen}
\usepackage{tikz}
\usepackage{appendix}
\usepackage{verbatim}
\usetikzlibrary{decorations.pathmorphing}
\usetikzlibrary{calc}
\usepackage{hyperref}

\input xy
\xyoption{all}

\tikzset{vert/.style={circle,fill,inner sep=0,
    minimum size=0.15cm,draw}, new/.style={}}

\renewcommand{\comment}[1]{}
\newcommand{\eq}{\begin{equation}}
\newcommand{\en}{\end{equation}}
\newcommand{\rr}{\mathbb{R}}

\newcommand{\norm}[1]{\left\lVert #1 \right\rVert}
\newcommand{\abs}[1]{\left\lvert #1 \right\rvert}

\newcommand{\iprod}[1]{\left\langle #1 \right\rangle }

\newcommand{\wass}{\mathbb{W}}
\newcommand{\Ent}{\mathrm{Ent}}

\newcommand{\tq}{\widetilde{q}}
\newcommand{\tp}{\widetilde{p}}
\newcommand{\tf}{\widetilde{f}}
\newcommand{\tmu}{\widetilde{\mu}}
\newcommand{\tnu}{\widetilde{\nu}}
\newcommand{\tchi}{\widetilde{\chi}}

\newcommand{\cost}{\mathbf{C}}
\newcommand{\Diri}{\mathrm{Diri}}
\newcommand{\ldiv}{\mathbf{L}}
\newcommand{\tnabla}{\widetilde{\nabla}}

\newcommand{\wt}[1]{\widetilde{#1}}
\newcommand{\support}{\mathrm{spt}}
\newcommand{\tLambda}{\widetilde{\Lambda}}

\newcommand{\bpi}{\boldsymbol{\pi}}
\newcommand{\interior}{\mathrm{int}}

\DeclareMathOperator*{\esssup}{ess\,sup}
\DeclareMathOperator*{\essinf}{ess\,inf}

\begin{document}

\theoremstyle{plain}
\newtheorem{thm}{Theorem}
\newtheorem{lemma}[thm]{Lemma}
\newtheorem{prop}[thm]{Proposition}
\newtheorem{cor}[thm]{Corollary}

\theoremstyle{definition}
\newtheorem{defn}{Definition}
\newtheorem{asmp}{Assumption}
\newtheorem{notn}{Notation}
\newtheorem{prb}{Problem}

\theoremstyle{remark}
\newtheorem{rmk}{Remark}
\newtheorem{exm}{Example}
\newtheorem{clm}{Claim}

\title[Limit of entropic cost]{On the difference between entropic cost and the optimal transport cost}

\author{Soumik Pal}
\address{University of Washington, Seattle}
\email{soumikpal@gmail.com}



\thanks{Many thanks to Matthias Erbar, Martin Huesmann, Jonathan Niles-Weed and Leonard Wong for very useful conversations. This research is partially supported by NSF grant DMS-1612483}

\date{\today}

\begin{abstract} Consider the Monge-Kantorovich problem of transporting densities $\rho_0$ to $\rho_1$ on $\rr^d$ with a strictly convex cost function. A popular regularization of the problem is the one-parameter family called the entropic cost problem. The entropic cost $K_h$, $h>0$, is significantly faster to compute and $h K_h$ is known to converge to the optimal transport cost as $h$ goes to zero. We are interested the rate of convergence. We show that the difference between $K_h$ and $1/h$ times the optimal cost of transport has a pointwise limit when transporting a compactly supported density to another that satisfies a few other technical restrictions. This limit is the relative entropy of $\rho_1$ with respect to a Riemannian volume measure on $\rr^d$ that measures the local sensitivity of the transport map. For the quadratic Wasserstein transport, this relative entropy is exactly one half of the difference of entropies of $\rho_1$ and $\rho_0$. In that case we complement the results of Adams et al. and others \cite{ADPZ11, DLR13, EMR15} who use gamma convergence. More surprisingly, we demonstrate that this difference of two entropies (plus the cost) is also the limit for the Dirichlet transport introduced by Pal and Wong \cite{PW18}. The latter can be thought of as a multiplicative analog of the Wasserstein transport and corresponds to a non-local operator. It hints at an underlying gradient flow of entropy, in the sense of Jordan-Kinderlehrer-Otto, even when the cost function is not a metric. The proofs are based on Gaussian approximations to Schr\"odinger bridges as $h$ approaches zero. 
\end{abstract}

\maketitle

\section{Introduction.} Let $\rho_0$ and $\rho_1$ be Borel probability density functions on $\rr^d$. We will assume throughout this paper that these are continuous and compactly supported. Following a standard abuse of notation, we will make no notational distinction between absolutely continuous measures and their densities. 

Consider the Monge-Kantorovich problem of transporting $\rho_0$ to $\rho_1$ optimally with respect to a cost function $c(x,y)=g(x-y)$ where $g:\rr^d \rightarrow [0, \infty)$ is strictly convex,  $g(0)=0$, $g(x)>0$ for $x\neq 0$ and the Hessian of $g$ at the origin, $\nabla^2 g(0)$, is invertible. More formally, let $\Sigma= \rr^d \times \rr^d$ and let $M_1(\Sigma)$ denote the set of Borel probability measures on $\Sigma$ equipped with the L\'evy metric of weak convergence. We will throughout denote $(X,Y)$ as a pair of $d$-dimensional random variables with law  in $M_1(\Sigma)$. For $\nu \in M_1(\Sigma)$, define the two projections, $\pi^0(\nu)$ and $\pi^1(\nu)$, to be the marginal laws of $X$ and $Y$, where the pair $(X,Y)$ is distributed as $\nu$. Given a pair of densities $(\rho_0, \rho_1)$, let $\Pi(\rho_0, \rho_1)$ be the subset of $M_1(\Sigma)$ such that if $\nu\in \Pi(\rho_0, \rho_1)$ then $\pi^0(\nu)=\rho_0$ and $\pi^1(\nu)=\rho_1$. We call this the set of couplings of $(\rho_0, \rho_1)$ and will frequently denote by the notation $X \sim \rho_0$ and $Y \sim \rho_1$.

The Monge-Kantorovich optimal transport cost for the initial density $\rho_0$ and the target density $\rho_1$ with cost $c(x,y)=g(x-y)$ is given by
\eq\label{eq:MKOT}
\wass_g(\rho_0, \rho_1):=\inf_{\nu \in \Pi(\rho_0, \rho_1)} \nu\left( g(x-y)\right).
\en
Here and throughout, if $\nu$ is a measure and $f$ on $\Sigma$ is $\nu$-integrable, the integral will be denoted by $\nu(f)$.
In other words, we want the coupling $\nu$ of $(\rho_0, \rho_1)$ that minimizes the expected cost $\nu(g(x-y))$. The coupling that achieves this optimal cost is called the optimal coupling. We say that the optimal coupling solves the Monge problem if it is of the form $(X,T(X))$, where $X \sim \rho_0$ and $T:\rr^d \rightarrow \rr^d$ is a measurable map that push-forwards $\rho_0$ to $\rho_1$. 
The particularly well-known case is when $g(x-y)=\frac{1}{2}\norm{x-y}^2$, for which we will denote $\wass_g(\rho_0, \rho_1)$ by $\frac{1}{2}\wass_2^2(\rho_0, \rho_1)$.

For a probability density $\rho$ (or, its corresponding measure) define its entropy by $\Ent(\rho):= \int \rho(x) \log \rho(x) dx$. This definition is the negative of the usual Shannon differential entropy. Entropy of a measure that is not absolutely continuous is taken to be $\infty$. Recall the notion of relative entropy of two densities $\rho$ with respect to $\rho'$
\[
H(\rho \mid \rho')=\int \rho(y) \log \frac{\rho(y)}{\rho'(y)} dy. 
\]
The above is nonnegative if both $\rho, \rho'$ are probability densities, but not in general. The following entropic regularization of the transport problem \eqref{eq:MKOT} has become quite popular, especially in connection to efficient computational algorithms \cite{Cut13, PC19}. 

\begin{defn}[Entropic regularization]
Fix a parameter $h > 0$ and define the entropic regularization of optimally transporting $\rho_0$ to $\rho_1$ with parameter $h$ as 
\eq\label{eq:entropicprime}
K'_h\left( \rho_0, \rho_1 \right) := \inf_{\nu \in \Pi(\rho_0, \rho_1)}\left[ \frac{1}{h}\nu\left( g(x-y)\right) + \Ent(\nu) \right]. 
\en
\end{defn}

The quantity $K'_h$ is closely related to minimizing a relative entropy. Since $\rho_0$ and $\rho_1$ are compactly supported, without affecting the transport problem, we can assume that $g(z)=\infty$ outside a suitably large compact ball. Then consider the joint probability density on $\Sigma$, 
\eq\label{eq:whatismuh}
\mu_h(x,y)= \rho_0(x) \frac{1}{\Lambda_h} \exp\left( -\frac{1}{h} g(x-y)\right), 
\en
where $\Lambda_h$ is the normalizing function $\int_{\rr^d}  \exp\left( -\frac{1}{h} g(x-y)\right) dy=\int_{\rr^d} \exp\left( -\frac{1}{h} g(z)\right) dz$. It is easily verifiable that, if $\nu\in \Pi(\rho_0, \rho_1)$, then  
\eq\label{eq:relaxcost}
H\left( \nu \mid \mu_h \right)= \frac{1}{h} \nu\left( g(x-y)\right) + \Ent(\nu) - \Ent(\rho_0) + \log \Lambda_h.  
\en

\begin{defn}[Entropic cost]
Fix a parameter $h > 0$ and define the entropic cost of optimally transporting $\rho_0$ to $\rho_1$ with parameter $h$ as 
\eq\label{eq:entropic}
K_h\left( \rho_0, \rho_1 \right) := \inf_{\nu \in \Pi(\rho_0, \rho_1)}H\left( \nu \mid \mu_h\right). 
\en
\end{defn}

Since the last two terms in \eqref{eq:relaxcost} do not depend on the coupling, 
\eq\label{eq:khkhp}
K_h(\rho_0, \rho_1) = K'_h(\rho_0, \rho_1) - \Ent(\rho_0) + \log \Lambda_h,
\en
while the minimizers of both $K_h$ and $K'_h$ are the same. 

As $h \rightarrow 0+$ one expects that the minimizer of the entropic cost problem will ``converge'' to the minimizer of \eqref{eq:MKOT}. Such results are known under regularity assumptions (see  \cite[Theorem 3.3]{L12}) and are made precise in terms of \textit{gamma convergence}. In passing, let us mention that the entropic regularization is frequently written in terms a different parameter $\lambda={1}/{h}$: $$h K'_h(\rho_0, \rho_1)=\inf_{\nu \in \Pi(\rho_0, \rho_1)}\left[ \nu\left( g(x-y)\right) + \frac{1}{\lambda}\Ent(\nu) \right].$$
Of course, this does not affect the minimizers. 

In this paper we are interested in the limit $\lim_{h\rightarrow 0+} \left[ K_h(\rho_0, \rho_1) - \frac{1}{h} \wass_g(\rho_0, \rho_1)\right]$.
Since the entropic cost is easier to compute, it is natural to ask for error bounds from the optimal cost. 

In the case of the quadratic Wasserstein cost, it is now a famous result (see \cite{ADPZ11} for dimension one and \cite{DLR13, EMR15} for higher dimensions) that 
\[
\lim_{h\rightarrow 0+} \left[ K_h(\rho_0, \rho_1) - \frac{1}{2h} \wass^2_2(\rho_0, \rho_1)\right]= \frac{1}{2}\left(\Ent(\rho_1) - \Ent(\rho_0) \right),
\]
where the above limit holds in the sense of gamma convergence. Whether such a limit holds pointwise for a given pair of $(\rho_0, \rho_1)$ was left open, and is resolved here as a corollary of our more general result that covers many other convex cost functions and more. To state our result we need a definition. 

It is proved in \cite{GM96} that for a cost function as above the optimal coupling is also the solution of the Monge problem, and, in fact, the following description of the optimal solution exists. For $c(x,y)=g(x-y)$ as above, define the class of $c$-concave functions as a function $\psi:\rr^d \rightarrow \rr\cup\{-\infty\}$ for which there exists a function $\xi:\rr^d \rightarrow \rr\cup \{-\infty\}$ such that $\psi(x) =\inf_{y\in \rr^d} \left[ g(x-y) - \xi(y)\right]$.

It has been shown in \cite{GM96} that, except on a set of dimension $(d-1)$, a $c$-concave function is differentiable, wherever it is finite, and it is twice differentiable almost everywhere in the sense of Alexandrov. Theorem 1.2 from \cite{GM96} shows that there is a $c$-concave function $\psi$ on $\rr^d$ such that the map 
\eq\label{eq:whatisxstar}
T(x)=x^*:= x- \left( \nabla g\right)^{-1} \circ \nabla \psi
\en
push-forwards $\rho_0$ to $\rho_1$. Additionally, the coupling $(X, X^*)$, $X \sim \rho_0$, achieves the infimum in \eqref{eq:MKOT}. This solution is unique, almost surely, on $\support(\rho_0)$, the support of $\rho_0$.
When $g(z)=\frac{1}{2}\norm{z}^2$, $\nabla g$ is the identity map and $\psi$ is $c$-concave is equivalent to the assertion that $\frac{1}{2}\norm{x}^2- \psi(x)$ is convex and lower semicontinuous in the usual sense. Thus $x\mapsto x^*$ is the well-known Brenier map.

Corresponding to every $g$ and $\psi$ there is a related concept of divergence. This was originally defined in \cite{PW18} in more generality. We restate the definition here limiting ourselves to this special case of $c(x,y)=g(x-y)$. The following definitions are standard and can be found in \cite[Chapter 1]{AGS08}. 

Distinguish between two copies of $\rr^d$ depending on whether the $x$ or the $y$ coordinate belongs there. This is required due to the possible asymmetry of $g(x-y)$. Let $\psi$ denote a $c$-concave function of $x$. One can define (see \cite[Definition 2.6]{GM96}) the $c$-superdifferential $\partial^c\psi$ to be the set of pairs $(x,y)\in \Sigma$ for which $\psi(v) \le \psi(x) + g(v-y) - g(x-y)$, for all $v\in \rr^d$. 
Given $x$, the set of all $y$ such that $(x,y)\in \partial^c \psi$ will be denoted by $\partial^c\psi(x)$. Define the c-dual $\psi^*$ as $\psi^*(y)= \inf_{x\in \rr^d} \left[ g(x-y) - \psi(x)\right]$. 
It turns out that 
\eq\label{eq:dualineq}
\psi(x) + \psi^*(y) \le c(x,y)=g(x-y),
\en 
with equality precisely when $(x,y)\in \partial^c\psi$. Finally from \cite[Proposition 3.4]{GM96}, we get that for $g$ satisfying our assumptions,  $\partial^c\psi(x)=\{x^*\}$ whenever $\psi$ is differentiable at $x$, which happens Lebesgue almost everywhere on the domain of $\psi$. 

We are going to assume the following regularity conditions. For any $w\in \rr^d$ and $0 < \theta < \pi$, define the cone $K(w,\theta)$ with direction $w$ and aperture $\theta$ by 
\[
K(w, \theta) =\{x \in \rr^d:\; \iprod{x, w} \ge \norm{x} \norm{w} \cos(\theta)\}.
\]

\begin{asmp}\label{asmp:density}
Assume that $\rho_0, \rho_1$ are continuous and compactly supported and $\rho_1$ is bounded away from zero on its support. Further assume that there is angle $\theta > 0$ and a constant $\delta > 0$ such that, for any $y \in \support(\rho_1)$, there exists $w \neq 0$ (depending on $y$), such that the shifted cone $y + K(w, \theta)$ intersection with $B_\delta(y)$ lies completely inside $\support(\rho_1)$. Here $B_\delta(y)$ is the open Euclidean ball of radius $\delta$ at $y$. Finally, assume that the Lebesgue measure of the boundary of $\support(\rho_1)$ is $0$.
\end{asmp}

The set $T(\support(\rho_0))$ is a set where $\rho_1$ is concentrated and the map $x \mapsto x^*$ is invertible almost everywhere. 
Let $\interior\left(\support(\rho_1)\right)$ denote the interior of the support of $\rho_1$. We will denote the set $T(\support(\rho_0))\cap \interior\left(\support(\rho_1)\right)$ by $\support_0(\rho_1)$. By Assumption \ref{asmp:density}, $\rho_1$ is concentrated on this set.  

\begin{defn}[c-divergence]\label{defn:cdiv}
For a $c$ concave function $\psi$, the divergence is a function from $\Sigma=\rr^d \times \rr^d$ to $[0, \infty]$. If $y \in \support(\rho_1)$ and $z=x^*$, then, define 
\[
D[y \mid z ]=D[y \mid x^*]:= g(x-y) - \psi(x) - \psi^*(y).
\]
For any other $y,z$, let $D[y \mid z]=\infty$. Equivalently, we can redefine $\psi(x)=-\infty$, for $x\notin \support(\rho_0)$ and $\psi^c(y)=-\infty$, for $y \notin \support(\rho_1)$.  
\end{defn}

Hence, divergence measures the so-called ``slackness'' in Linear Programing. A transport interpretation of divergence is provided in \cite{PW16} where it is shown to be a measure of the error in transporting $x$ to $y$ instead of the optimal $x$ to $x^*$. See a more geometric analysis in \cite{W17}. For the quadratic cost, this notion of divergence coincides with the well-known Bregman divergence and is a fundamental quantity in Information Geometry \cite{A16}.

Clearly $D[y\mid x^*]\ge 0$ by \eqref{eq:dualineq} and, if $\partial^c\psi(x)=\{x^*\}$, then $D[y\mid x^*]$ is exactly zero if and only if $y=x^*$. It follows that $D[\cdot \mid x^*]$ must be locally quadratic in a neighborhood of such an $x^*$. 

\begin{asmp}\label{asmp:continuity} Assume that for any $\delta >0$, 
\eq\label{eq:divsublev}
\inf_{z\in \support(\rho_1)} \inf_{y \in B^c_\delta(z)} D[y \mid z] >0,
\en   
where $B^c_\delta(z)$ is the complement of the open Euclidean ball around $z$ of radius $\delta$. Assume also that the following quadratic approximation to the divergence holds uniformly on $\support_0(\rho_1)$. For any $\epsilon >0$, there is a $\delta>0$ such that for all $z \in \support_0(\rho_1)$ and all $y\in \support(\rho_1)$ such that $\norm{y-z}< \delta$, we have  
\eq\label{eq:remain1}
R[y\mid z] := D[y\mid z]- \frac{1}{2}(y-z)^T A(z) (y-z) = \epsilon \abs{y-z}^2,
\en
for some family of matrices $A(z)$ that is continuous a.s. on $\support_0(\rho_1)$. The family $A(\cdot)$ is assumed to be uniformly elliptic in the sense that there exists an $\varepsilon>0$ such that the eigenvalues of $A(z)$ lie in $(\varepsilon, 1/\varepsilon)$ for a.e. $z \in \support_0(\rho_1)$. In particular, the family of inverses $A^{-1}(z)$, $z\in \support_0(\rho_1)$, exists and is continuous almost surely. \footnote{For the quadratic Wasserstein transport the assumption is asking for continuity and uniform ellipticity of the Hessian matrix of the dual Kantorovich potential function a.e..}

We also assume that a similar Taylor approximation holds for the convex function $g$ at the origin:
\eq\label{eq:remain2}
r(z):=g(z) - \frac{1}{2} z^T \nabla^2g(0) z = o\left( \abs{z}^2\right), \quad \text{as} \; z \rightarrow 0, 
\en
where the matrix $\nabla^2 g(0)$ is assumed to be invertible. 
\end{asmp}

Consider the inverse matrix $A^{-1}(z)$ as a Riemannian matrix derived from the divergence.
Define the Riemmanian volume measure on $\rr^d$ by 
\eq\label{eq:riemannvol}
\frac{d\mu_g(y)}{dy}= \rho_1(y) \sqrt{\det(A^{-1}(y))}, \quad \text{a.e.}\; y \in \rr^d,
\en
then the relative entropy of $\rho_1$ with respect to $\mu_g$ is given by $H(\rho_1 \mid \mu_g)=\frac{1}{2} \int \rho_1(y) \log \det\left( A(y)\right) dy$.

\begin{thm}\label{thm:mainthmw} Suppose Assumptions \ref{asmp:density} and \ref{asmp:continuity} are satisfied. Then, the following convergences hold.
\eq\label{eq:entasymp}
\begin{split}
\lim_{h\rightarrow 0+}& \left( K_h(\rho_0, \rho_1) - \frac{1}{h} \wass_g(\rho_0, \rho_1) \right) = -\frac{1}{2} \log \det \nabla^2 g(0) + H(\rho_1 \mid \mu_g).\\
\lim_{h\rightarrow 0+}& \left( K'_h(\rho_0, \rho_1) - \frac{1}{h} \wass_g(\rho_0, \rho_1) + \frac{d}{2}\log(2\pi h) \right) = \Ent(\rho_0) + H(\rho_1 \mid \mu_g).
\end{split}
\en 
In particular, for the quadratic Wasserstein cost $g(x-y)=\frac{1}{2}\norm{x-y}^2$, 
\eq\label{eq:entasympw}
\begin{split}
\lim_{h\rightarrow 0+} &\left( K_h(\rho_0, \rho_1) - \frac{1}{2h} \wass^2_2(\rho_0, \rho_1) \right) = \frac{1}{2}\left( \Ent(\rho_1) - \Ent(\rho_0) \right).\\
\lim_{h\rightarrow 0+} & \left( K'_h(\rho_0, \rho_1) - \frac{1}{2h} \wass^2_2(\rho_0, \rho_1) + \frac{d}{2}\log(2\pi h)\right) = \frac{1}{2}\left( \Ent(\rho_1) + \Ent(\rho_0) \right).
\end{split}
\en
\end{thm}

The proof of this theorem follows from a Gaussian approximation which is best stated in the quadratic Wasserstein case. Consider Brownian motion $\{X_t,\; 0\le t \le h\}$ in $\rr^d$ ``conditioned'' to have initial distribution $\rho_0$ and terminal distribution $\rho_1$ at time $h$. Such a process is called a Schr\"odinger bridge between $\rho_0$ and $\rho_1$. See \cite{leonardsurvey} and \cite{conforti2018}. It is intuitive that the law of the vector $(X_0, X_h)$ is the minimizer of the entropic cost $K_h$. Hence, if we can exactly describe its joint density, we can compute $K_h$. Exactly describing such a bridge is very hard although can be done abstractly as an $(f,g)$ transform. See \cite[Section 3]{leonardsurvey}. The core of our proof is to show that, for $h \approx 0$, the bridge is approximately the following. Sample $X_0$ from $\rho_0$. Given $X_0=x$, sample $X_h$ from a Gaussian distribution with mean $x^*$ and covariance $h A^{-1}(x^*)$. Of course, the distribution of $X_h$, say $\rho_1^h$, obtained this way is not exactly $\rho_1$. However, as $h \rightarrow 0+$, the approximation is tight. This Gaussian approximation also gives an interpretation of the matrix $A^{-1}(x^*)$. If $A^{-1}(x^*)$ is large, one can perturb $X_h$ more, relatively speaking, around $x^*$ without paying too much transportation cost. However, if $A^{-1}(x^*)$ is small, $X_h$ has a relatively smaller fluctuation around $x^*$. Hence the Riemannian volume measure $\mu_g$ can be thought of as a measure of sensitivity of the Monge map.

\subsection{The Dirichlet transport}\label{sec:dirichletintro} The quadratic Wasserstein case \eqref{eq:entasympw} is very special since the limit is the difference of the two entropies. See the next section for its connection to heat equation and gradient flow of entropy. Rather surprisingly, a similar limit appears for a completely different transport cost called the Dirichlet transport problem that can be thought of as a multiplicative analog of the Wasserstein transport \cite{PW18}. Let $n \geq 2$ be an integer and consider the open simplex 
\[
\Delta_n=\left\{ (p_1, p_2, \ldots, p_n):\; p_i > 0,\; \forall\; i, \; p_1 + \ldots + p_n=1\right\}.
\]
The unit simplex has an abelian group structure with a multiplicative group operation
\begin{equation} \label{eqn:simplex.group.operation}
p \odot q := \left( \frac{p_iq_i}{\sum_{j = 1}^n p_jq_j} \right)_{1 \leq i \leq n}, \quad p, q \in \Delta_n, \quad p^{-1} := \left( \frac{1/p_i}{\sum_{j = 1}^n 1/p_j}\right)_{1 \leq i \leq n}. 
\end{equation}
The identity element is the barycenter $\overline{e} := \left( \frac{1}{n}, \ldots, \frac{1}{n}\right)$.
The $\odot$ operation plays an analogous role as vector addition in the Wasserstein transport. In particular, the sign changes that often appear in standard optimal transport definitions will be replaced here by inversion $p \mapsto p^{-1}$ which takes a little bit of time to get used to.

Let $c: \Delta_n \times \Delta_n \rightarrow [0, \infty)$ be the cost function defined (see \cite[Lemma 6]{PW18})
\eq\label{eq:costDirichlet}
c(p,q)= \log\left(\frac{1}{n} \sum_{i=1}^n \frac{q_i}{p_i} \right) - \frac{1}{n} \sum_{i=1}^n \log\frac{q_i}{p_i}=H\left( \overline{e} \mid q\odot p^{-1}\right), 
\en
where $H$ is the discrete relative entropy defined on $\Delta_n \times \Delta_n$ by $H\left(p \mid q\right) := \sum_{i = 1}^n p_i \log \frac{p_i}{q_i}$. Clearly, $c(p, q) \geq 0$ for all $p, q$, and $c(p, q) = 0$ only if $p = q$. It is clear that the cost function is not symmetric in $p$ and $q$, although $c(p,q)=c(q^{-1}, p^{-1})$.

The variable $\pi:=q\odot p^{-1}$ will play an important role throughout this paper. Following our previous works \cite{PW14, PW16, P17} we call $\pi$ the {\it portfolio vector}. Note that $p = q$ (i.e., $c(p, q) = 0$) if and only if the portfolio vector $\pi$ is equal to the barycenter $\overline{e}$. Since $H(p \mid q)$ is a convex function of $q$ for fixed $p$, our cost function $H(\overline{e}\mid q\odot p^{-1})$ is similar to the cost $g(y-x)$, for $x,y\in \rr^d$, where the group operation $+$ on $\rr^d$ has been replaced by $\odot$. It is also possible by a change of coordinates (see Section \ref{sec:Dirichlet}) to reduce this cost to $g(x-y)$ for a convex $g$. However, the group structure $\odot$, which is critical towards understanding this transport, is lost in the process.

As before, given densities $\rho_0, \rho_1$, compactly supported on $\Delta_n$, consider the cost $\cost(\rho_0, \rho_1)$ of transporting $\rho_0$ to $\rho_1$ with respect to cost $c$. It is clear that $\cost(\rho_0, \rho_1)$ is not a metric since it is asymmetric in $\rho_0$ and $\rho_1$. The regularity theory for this transport has been recently studied in \cite{KZ19} where it has been shown that this cost function satisfies the Ma-Trudinger-Wang condition \cite{MTW}. Hence the transport map is smooth when $\support(\rho_1)$ is convex.

The name Dirichlet transport comes from its deep connection to the Dirichlet distribution as shown in \cite{PW18}. For any $\lambda > 0$, consider the symmetric Dirichlet distribution $\Diri(\lambda)$ with parameters $\left( \lambda/n, \ldots, \lambda/n\right)$ on $\Delta_n$ with density 
\eq\label{eq:diridensity}
\frac{\Gamma(\lambda)}{\left(\Gamma(\lambda/n)\right)^n} \prod_{i=1}^n p_i^{\lambda/n -1}, \quad p_n:= 1-\sum_{i=1}^{n-1} p_i.
\en
The density is with respect to the $(n-1)$ dimensional Lebesgue measure on the open set $\{ (p_1, p_2, \ldots, p_{n-1}),\; p_i > 0,\; \sum_{i=1}^{n-1} p_i < 1  \}$. 

To calculate entropy consider the Haar measure on $\Delta_n$ with respect to the group operation $\odot$. This measure is the sigma-finite Dirichlet measure $\Diri(0)$ with a density $\nu_0(p)=\prod_{i=1}^n p_i^{-1}$. For any probability density $\rho$ on $\Delta_n$, define
\[
\Ent_0(\rho) := \rho \left( \log \frac{\rho}{\nu_0} \right).
\]
Let $\mu_h$ denote the joint density on $\Delta_n \times \Delta_n$ where we sample $X \sim \rho_0$ and, given $X=p$, $Y$ has the distribution of $p \odot G_h$, for an independent Dirichlet$(1/h)$ random variable $G_h$. Define the entropic cost of coupling $(\rho_0, \rho_1)$ with parameter $h$ as
\eq\label{eq:entropicdirichlet}
K_h(\rho_0, \rho_1) = \inf_{\nu \in \Pi(\rho_0, \rho_1)}H\left( \nu \mid \mu_h\right). 
\en
Given $\rho_0, \rho_1$, the Monge solution exists and is unique and there is a corresponding notion of divergence. The details can be found in \cite{PW18} (see Section \ref{sec:Dirichlet}).

\begin{thm}\label{thm:mainthm2}
Suppose that Assumption \ref{asmp:densityc} and \ref{asmp:continuityc} below (similar to Assumption \ref{asmp:density} and \ref{asmp:continuity}) are satisfied. Then,
\[
\lim_{h\rightarrow 0+} \left( K_h(\rho_0, \rho_1) - \left(\frac{1}{h} - \frac{n}{2} \right) \cost(\rho_0, \rho_1) \right) = \frac{1}{2}\left( \Ent_0(\rho_1) - \Ent_0(\rho_0) \right).
\]
\end{thm}


\subsection{Discussion on the gradient flow of entropy}
Consider the heat equation with initial condition $\rho_0$: 
\[
\partial_t \rho(t,x)= \frac{1}{2}\Delta_x \rho(t,x), \quad 0\le t \le 1,\; x\in \rr^d, \quad \rho(0,\cdot)=\rho_0.  
\]
Here $\Delta_x$ is the Laplacian in $x$. Let $\rho(t)$ denote the density $\rho(t, \cdot)$. 

In  \cite{JKO98} Jordan, Kinderlehrer, and Otto (JKO) interprets the heat equation as a gradient flow entropy in the $\wass_2$ metric space. The idea comes from the following discretization scheme that they propose. Let $h>0$ be the step size. Staring with $\rho^{(0)}=\rho_0$,
\eq\label{eq:eulerdisc}
\text{determine $\rho^{(k)}$ that minimizes}\quad \frac{1}{2h} \wass_2^2(\rho, \rho^{(k-1)}) + \left( \Ent(\rho) - \Ent(\rho^{(k-1)}) \right).
\en
Define the piecewise constant interpolation 
\[
\rho_h(t)=\rho^{(k)}, \quad \text{for}\; t \in [kh, (k+1) h)\; \text{and}\; k=0,1,2,\ldots.  
\]
Then, \cite{JKO98} shows that, as $h\downarrow 0$, $\rho_h(t)$ converges to $\rho(t)$ weakly in $\mathbf{L}^1$. Their theorem is somewhat more general since it includes Fokker-Planck equations, but that extension from the heat equation is not very hard. 

In \cite{ADPZ11} the authors developed an alternative approach via the entropic cost function. Recall the discussion in the paragraph above Section \ref{sec:dirichletintro}. Let $\mu_h(x,y)$ be the joint density of $\left( X_0, X_h\right)$. Then it follows that $K_h(\rho_0, \rho_1) := \inf_{\nu \in \Pi(\rho_0, \rho_1)} H\left( \nu \mid \mu_h\right)$. Then the authors argued (dimension one, compactly supported densities) that \eqref{eq:entasympw} holds in the sense of gamma convergence. Since the solution to the heat equation could be seen as successively minimizing the entropic cost over a discrete set of points $h, 2h, 3h, \ldots$, it gives credence to the idea that, for small values of $h$, the Euler discretization scheme \eqref{eq:eulerdisc} coverges to the solution of the heat equation. Making a rigorous theory of such gradient flow of entropy is highly non-trivial and can be found in the textbook \cite{AGS08}. Very recently, a pathwise version of this gradient flow result has appeared in \cite{KST}.

Theorem \ref{thm:mainthm2} regarding the Dirichlet transport hints at a possible extension of the theory of gradient flow of entropy on metric spaces of probabilities. $\mu_h$ corresponds to a stochastic process that evolves by successive multiplications $\odot$ by independent symmetric Dirichlet distributed random variables. Such an evolution is non Fokker-Planck and, in fact, non local, since it proceeds by jumps. However, the entropic cost admits a similar asymptotic formula as that of quadratic Wasserstein. Thus, one naturally expects a theory of ``gradient flow of entropy'' even when the cost of transport $\cost$ is not a metric between probabilities. For non-local operators corresponding to stochastic processes such as jump processes and Markov chains, a rather important theory of gradient flow of entropy has been developing steadily. See articles \cite{Maas11, EM12, Erbar14, EM14, dietert2015, EFLS}. However, the cost function considered in these papers are metrics constructed by generalizations of the Benamou-Brenier formula.   

As a final comment, both quadratic Wasserstein and the Dirichlet transports are related to probability distributions (Gaussian and Dirichlet, respectively) that form an exponential family \cite{A16}. In \cite{P17} it is shown that the transport for all exponential families share similar features including the structure of the Kantorovich potential and the Monge solution. It seems worthwhile to extend the current analysis as well.

\section{General convex cost and the proof of Theorem \ref{thm:mainthmw}.} 

Let $\rho_0, \rho_1$ satisfy Assumption \ref{asmp:density}. Fix $h>0$. Recall the joint density $\mu_h$ from \eqref{eq:whatismuh} which is the product of $\rho_0(x)$ and a transition densities on $\Sigma$:
\eq\label{eq:transition}
p_h(x,y):=\frac{1}{\Lambda_h} \exp\left( -\frac{1}{h} g(x-y)\right), \quad \Lambda_h:= \int_{\rr^d}  \exp\left( -\frac{1}{h} g(z)\right) dz.
\en
Note that $\Lambda_h$ is finite since $g$ is infinity outside a compact ball (determined by the supports of $\rho_0, \rho_1$). 
Recall $\psi$ be the $c$-concave function that determines the Monge map as in \eqref{eq:whatisxstar} transporting $\rho_0$ to $\rho_1$. Let $\psi^*$ denote its c-dual. Assume, as before, $\psi(x)=-\infty$, for $x\notin \support(\rho_0)$ and $\psi^c(y)=-\infty$, for $y \notin \support(\rho_1)$. The following \textit{exponential tilting} $p_h$ is of central importance. 
\eq\label{eq:transitiontilt}
\begin{split}
\tp_h(x,y)&:= \frac{1}{Z_h(x)} \exp\left( \frac{1}{h} \left( \psi(x) + \psi^*(y)\right)\right) p_h(x,y)\\
&= \frac{1}{Z_h(x) \Lambda_h} \exp\left( -\frac{1}{h} D[y\mid x^*]\right), \quad \text{from Definition \ref{defn:cdiv}}\\
&=\frac{1}{ \tLambda_h(x)} \exp\left( -\frac{1}{h} D[y\mid x^*]\right), \quad \text{say}.
\end{split}
\en
Here $\tLambda_h(x):= Z_h(x)\Lambda_h$ and $Z_h(x)$ is the appropriate normalizing function
\[
Z_h(x):= \frac{1}{\Lambda_h} \int \exp\left( -\frac{1}{h} D[y\mid x^*] \right)dy,
\]
which is clearly finite since $\support(\rho_1)$ is compact.

\begin{lemma}\label{lem:partitionfnasymp}
Under Assumptions \ref{asmp:density} and \ref{asmp:continuity} the following statements hold. 
\begin{enumerate}[(i)]
\item The normalizing constant $\Lambda_{h}$ converges to the following limit.
\[
\lim_{h \rightarrow 0+}  \frac{\Lambda_{h}}{(2\pi h)^{d/2}} =\frac{1}{ \sqrt{ \det\left(\nabla^{2}g(0) \right)}}.
\]

\item The normalizing constant $\tLambda_{h}(x)$ has the following limit for $\rho_0$-a.e. $x$ such that $x^* \in \support_0(\rho_0)$.
\[
\lim_{h \rightarrow 0+}  \frac{\tLambda_{h}(x)}{(2\pi h)^{d/2}} = \frac{1}{\sqrt{\abs{J}(x^*)}}, \quad \abs{J}(x^*):=\det\left( A(x^*)\right),
\]
where the matrix valued function $A(\cdot)$ is defined in Assumption \ref{asmp:continuity}. 

\item Recall the Riemannian volume measure from \eqref{eq:riemannvol}. Then, 
\[
\lim_{h\rightarrow 0+}  -\int \log Z_h(x) \rho_0(x) dx= -\frac{1}{2}\log \det \nabla^2 g(0) + H\left( \rho_1 \mid \mu_g\right), 
\]
where $H\left( \rho_1 \mid \mu_g\right)=\frac{1}{2}\int \rho_1(y)\log \abs{J}(y) dy$.
In particular, when $g(z)=\frac{1}{2}\norm{z}^2$ (inside the compact set where it is finite), 
\[
\begin{split}
\lim_{h\rightarrow 0+}  -\int \log Z_h(x) \rho_0(x) dx &= \frac{1}{2} \left( \Ent(\rho_1) - \Ent(\rho_0)\right).
\end{split}
\]
\item Let $\tmu_h(x,y)$ denote the joint density $\rho_0(x)\tp_h(x,y)$.
Let $\rho_1^h$ denote the density of $Y$ under $\tmu_h$. 
Then $\lim_{h\rightarrow 0+} \rho_1^h(y)= \rho_1(y)$ for all $y \in \support(\rho_1)$. 
\end{enumerate}
\end{lemma}

The long proof of this lemma is broken down in separate parts.

\begin{proof}[Proof of Lemma \ref{lem:partitionfnasymp} (i)] For notational simplicity let $V$ temporarily denote the invertible Hessian matrix $\nabla^2 g(0)$. Let $q_h(\cdot)$ denote the Gaussian density on $\rr^d$ with mean $0$ and covariance matrix $h V^{-1}$. That is, 
\[
q_h(z)=\frac{\det\left( V^{1/2}\right)}{\left( 2\pi h\right)^{d/2}} \exp\left( -\frac{1}{2 h} z^T V z\right), \quad z \in \rr^d. 
\] 
Recall $r(\cdot)$ from \eqref{eq:remain2}. Then, by a change of measure,
\eq\label{eq:2bnd}
\begin{split}
\Lambda_h \frac{\det\left( V^{1/2}\right)}{\left( 2\pi h\right)^{d/2}} &=\int_{\rr^d} \exp\left( -\frac{1}{h} r(z) \right) q_h(z) dy.
\end{split}
\en
We show that the right side converges to one as $h \rightarrow 0+$. 

By \eqref{eq:remain2} and that $V$ is non-singular, for any $\epsilon \in (0,1)$, there exists a $\delta >0$ such that 
\eq\label{eq:localthird1}
\abs{r(z)} \le \frac{\epsilon}{2} z^T V z,
\en
for all $z \in B_\delta(0)$, the Euclidean ball of radius $\delta$ around $0$. Consider the decomposition
\eq\label{eq:localdecomp}
\begin{split}
\int e^{-r(z)/ h} q_h(z) dz &= \int_{B_\delta(0)} e^{-r(z)/h} q_h(z) dz + \int_{B^c_\delta(0)} e^{-r(z)/h} q_h(z) dz.
\end{split}
\en

We will analyze the two terms separately. For the first term, 
\[
\begin{split}
 \int_{B_\delta(0)} & e^{-r(z)/h} q_h(z) dz \le  \int_{B_\delta(0)} \exp\left(\frac{\epsilon}{2h} z^T V z\right) q_h(z) dz \le \int_{\rr^d} \exp\left(\frac{\epsilon}{2 h} z^T V z\right) q_h(z) dz\\
 &\le \frac{\det(V^{1/2}) }{\sqrt{(2\pi  h)^d}} \int_{\rr^d} \exp\left[ -\frac{(1-\epsilon)}{2 h} z^T V z\right] dy  \le \left(1 - \epsilon \right)^{-d/2}. 
\end{split}
\]

For the second term on the right side of \eqref{eq:localdecomp}, reverse the change of measure.
\[
\begin{split}
\int_{B_\delta^c(0)} e^{-r(z)/h} q_h(z) dz = \frac{\det(V^{1/2})}{\sqrt{(2\pi h)^d}} \int_{B_\delta^c(0)} e^{-g(z)/h} dz.
\end{split}
\] 
Since $g$ is strictly convex, it has a positive infimum on $B_\delta^c(0)$. Moreover, $g=\infty$ outside a bounded domain. Hence, there exists positive constants $M, m$ such that the above expression is bounded above by 
\[
M \frac{\det(V^{1/2})}{\sqrt{(2\pi h)^d}} e^{-m/h}
\]
which tends to zero as $h\rightarrow 0+$. Thus, combining the two bounds, 
\[
\int e^{-r(z)/h} q_h(z) dz \le \left(1 - \epsilon \right)^{-d/2} + \epsilon, \quad \text{for all small enough $h>0$},
\]
which shows an upper bound on the right side of \eqref{eq:2bnd}.

For the lower bound, discard the integrand \eqref{eq:localdecomp} outside $B_\delta(0)$. For all small enough $h>0$, $q_h(\cdot)$ puts its almost entire mass in $B_\delta(0)$. Hence, for all small enough $h$,
\[
\begin{split}
\int &e^{-r(z)/ h} q_h(z) dz \ge  \int_{B_\delta(0)} e^{-r(z)/ h} q_h(z) dz \ge  \int \exp\left( - \frac{\epsilon}{2h} z^T V z \right) q_h(z) dz - \epsilon\\
&\ge  \int \frac{\det(V^{1/2})}{\sqrt{(2\pi h)^d}}\exp\left[ -\frac{(1+\epsilon)}{2h} z^T V z\right] dz - \epsilon \ge \left(1 + \epsilon \right)^{-d/2} - \epsilon. 
\end{split}
\]
Thus we get that 
\[
(1+\epsilon)^{-d/2} - \epsilon \le \liminf_{h\rightarrow 0+}\Lambda_h\frac{\det\left( V^{1/2}\right)}{\left( 2\pi h\right)^{d/2}}\le \limsup_{h\rightarrow 0+}\Lambda_h\frac{\det\left( V^{1/2}\right)}{\left( 2\pi h\right)^{d/2}} \le (1-\epsilon)^{-d/2} + \epsilon. 
\]
As we now take $\epsilon \downarrow 0$, we recover our claim. 
\end{proof}

\begin{proof}[Proof of Lemma \ref{lem:partitionfnasymp} (ii)] The proof of part (ii) is very similar to that of part (i) except at every step we will require an estimate over $x\in \support(\rho_0)$. For an $x$ such that $A(x^*)$ exists and the related conditions in Assumption \ref{asmp:continuity} hold, define the Gaussian probability transition kernel
\[
\tq_h(x,y)= \frac{\det(A^{1/2}(x^*))}{(2\pi h)^{d/2}} \exp\left( -\frac{1}{2h} (y-x^*)^T A(x^*) (y-x^*) \right), \; y \in \rr^d,
\]
which has mean $x^*$ and covariance $hA^{-1}(x^*)$. Then
\eq\label{eq:2bnd1}
\tLambda_h(x) \frac{\det(A^{1/2}(x^*))}{(2\pi h)^{d/2}}= \int_{\rr^d} \exp\left( - \frac{1}{h}R[y \mid x^*] \right) \tq_h(x,y) dy, 
\en
where the remainder term $R[y \mid x^*]$ is defined in \eqref{eq:remain1}. We show that the right side converges to one as $h \rightarrow 0+$. 

Since the analysis is similar to the proof of (i), it suffices to highlight just the main differences. First, by \eqref{eq:remain1} and the assumption that $A(x^*)$ has eigenvalues uniformly bounded away from zero, one can find a $\delta>0$ such that for all $x\in \support(\rho_0)$,
\eq\label{eq:estryx}
\abs{R[y \mid x^*]} \le \frac{\epsilon}{2} (y-x^*)^T A(x^*) (y-x^*), \quad y \in B_\delta(x^*) \cap \support(\rho_1). 
\en
If $y \notin \support(\rho_1)$, by definition $R[y \mid x^*]=\infty$. So the bound 
\[
\int_{B_\delta(x^*)} e^{-\frac{1}{h}R[y\mid x^*]} \tq_h(x,y) dy \le (1-\epsilon)^{-d/2}
\]
holds as before in (i). On $B_\delta^c(x^*)$, the infimum of $D[y \mid x^*]$ is positive by \eqref{eq:divsublev} and the divergence is infinity  outside a compact set that can be chosen uniformly for all $x\in \support(\rho_0)$. Therefore the complete upper bound holds as in (i) uniformly over our choice of $x$ in the sense that, for small enough $h>0$, 
\[
\esssup_{x\in \support(\rho_0)} \int_{\rr^d} \exp\left( - R[y \mid x^*] \right) \tq_h(x,y) dy \le (1-\epsilon)^{-d/2} + \epsilon. 
\]

We have to be slightly more careful in the proof of lower bound. Fix an $x$ such that $x^*\in \support_0(\rho_0)$. First, we derive a uniform lower bound. Due to the assumed cone condition in Assumption \ref{asmp:density}, there exists $w$ such that for all small enough $\delta >0$, $B_\delta(x^*) \cap (x^* + K(w, \theta)) \subseteq \support(\rho_1)$. Hence,
\[
\int_{B_\delta(x^*)} e^{-\frac{1}{h}R[y\mid x^*]} \tq_h(x,y) dy \ge \int_{B_\delta(x^*) \cap (x^* + K(w, \theta)) } e^{-\frac{1}{h}R[y\mid x^*]} \tq_h(x,y) dy
\]
Thus, on the set $B_\delta(x^*) \cap (x^* + K(w, \theta))$, the estimate \eqref{eq:estryx} is valid. Thus,
\eq
\begin{split}
\int_{B_\delta(x^*)}& e^{-\frac{1}{h}R[y\mid x^*]} \tq_h(x,y) dy 
\ge \int_{B_\delta(x^*)\cap (x^* + K(w, \theta))} \exp\left( -\frac{\epsilon}{2h} (y-x^*)^T A(x^*) (y-x^*)\right) \tq_h(x,y) dy\\
 \ge &\frac{\det A^{1/2}(x^*)}{(2\pi h)^{d/2}}\int_{B_\delta(x^*)\cap (x^* + K(w, \theta))}  \exp\left( -\frac{(1+\epsilon)}{2h} (y-x^*)^T A(x^*) (y-x^*)\right)  dy. 
\end{split}
\en 
Due to the assumed uniform ellipticity of the family of matrices $A(\cdot)$, the last integral is uniformly bounded in $x$. Hence, there is a uniform lower bound
\[
\essinf_{x \in \support(\rho_0)}\int_{B_\delta(x^*)} e^{-\frac{1}{h}R[y\mid x^*]} \tq_h(x,y) dy \ge c_0,\quad \text{for all $h$ small}, 
\]
for some $c_0 >0$. 

On the other hand, since $x^* \in \support_0(\rho_1)$, it is also in $\interior\left( \support(\rho_1)\right)$. Hence, for all small enough $\delta>0$, with a choice depending on $x$, the entire ball $B_\delta(x^*)\subseteq \support(\rho_1)$. Moreover, for all small enough $h>0$, $\tq_h(x, \cdot)$ puts almost entire mass inside $B_\delta(x^*)$. Hence,
\eq
\begin{split}
\int_{B_\delta(x^*)} &\exp\left( -\frac{\epsilon}{2h} (y-x^*)^T A(x^*) (y-x^*)\right) \tq_h(x,y) dy\\
 \ge \int_{\rr^d} & \exp\left( -\frac{\epsilon}{2h} (y-x^*)^T A(x^*) (y-x^*)\right) \tq_h(x,y) dy - \epsilon. 
\end{split}
\en 
The rest follows exactly as in (i). 

Hence, given $\epsilon >0$, there exists $h(\epsilon)>0$ such that for all $0< h < h(\epsilon)$, the following bound holds for $\rho_0$ a.e. $x$:
\[
\begin{split}
(1+\epsilon)^{-d/2} - \epsilon &\le \liminf_{h\rightarrow 0+} \tLambda_h(x)\frac{\det\left( A^{1/2}(x^*)\right)}{\left( 2\pi h\right)^{d/2}}\\
\le& \limsup_{h\rightarrow 0+} \tLambda_h(x)\frac{\det\left( A^{1/2}(x^*)\right)}{\left( 2\pi h\right)^{d/2}} \le (1-\epsilon)^{-d/2} + \epsilon. 
\end{split}
\]  
Taking $\epsilon \downarrow 0$ proves (ii). Additionally, there are uniform upper and lower bounds for $\tLambda_h(x)\frac{\det\left( A^{1/2}(x^*)\right)}{\left( 2\pi h\right)^{d/2}}$ that are valid for all $x$ such that $x^* \in \support_0(\rho_1)$. 
\end{proof}

\begin{proof}[Proof of Lemma \ref{lem:partitionfnasymp} (iii)] Since $Z_h(x) = \tLambda_h(x)/\Lambda_h$, from the convergence statements in parts (i) and (ii), and the uniform bounds mentioned in the last paragraph, it follows by an application of the Dominated Convergence Theorem that 
\eq\label{eq:partitionintasymp}
\begin{split}
\lim_{h\rightarrow 0+} \int \log Z_h(x) \rho_0(x) dx &= \frac{1}{2} \log \det \nabla^2 g(0) - \frac{1}{2}\int \log \abs{J}(x^*) \rho_0(x) dx\\
&=\frac{1}{2} \log \det \nabla^2 g(0) - \frac{1}{2} \int \log \abs{J}(y) \rho_1(y) dy.
\end{split}
\en
This proves our claim.

Now consider the special case $g(x-y)=\frac{1}{2}\norm{x-y}^2$, the quadratic Wasserstein cost. By Brenier's theorem, there exists a convex function $\varphi$ such that $x^*=\nabla \varphi(x)$. In terms of this function $\varphi$ and its convex conjugate $\varphi^*$, we can write 
\[
\psi(x)= \frac{1}{2} \norm{x}^2 - \varphi(x), \quad \psi^*(y) = \frac{1}{2} \norm{y}^2 - \varphi^*(y).
\] 
Thus, $D[y \mid x^*]=  \varphi(x) + \varphi^*(y) - \iprod{x,y}$ is the Bregman divergence of the convex function $\varphi$.
Hence 
\[
\det\left( \nabla^2 g(0)\right)=1, \quad A(y)=\nabla^2 \varphi^*(y), \quad \abs{J}(y)= \det\left( \nabla^2 \varphi^*(y)\right).
\]

The rest is a computation similar to \cite[Theorem 4.4]{M97}. Use the fact that the push-forward of $\rho_1$ by $\nabla \varphi^*$ gives us $\rho_0$. Thus, by the change-of-variable formula
\[
\abs{J}(x^*) = \frac{\rho_1(x^*)}{\rho_0(x)}, \; \text{or}\; \log \abs{J}(x^*) = \log \rho_1(x^*) - \log \rho_0(x).  
\]
Thus
\eq\label{eq:hessianentropy}
\begin{split}
\Ent(\rho_0)&= \int \rho_0(x) \log \rho_0(x) dx = \int \rho_1(y) \log \frac{\rho_1(y)}{\abs{J}(y)} dy\\
&= \Ent(\rho_1) - \int \rho_1(y) \log \abs{J}(y) dy.
\end{split}
\en
Therefore, $ \int \rho_1(y) \log \abs{J}(y) dy=\Ent(\rho_1)- \Ent(\rho_0)$, and, hence, from \eqref{eq:partitionintasymp},
\[
\begin{split}
\lim_{h\rightarrow 0+} &-\int \log Z_h(x) \rho_0(x) dx=- \frac{1}{2} \log (1) + \frac{1}{2} \int \log \abs{J}(y) \rho_1(y) dy \\
&=\frac{1}{2} \left( \Ent(\rho_1) - \Ent(\rho_0)\right).
\end{split}
\]
This completes the proof of the statement. 
\end{proof}

\begin{proof}[Proof of Lemma \ref{lem:partitionfnasymp} (iv)]
Consider the explicit expressions for $\rho_1^h(y)$ obtained by integrating $\tmu_h$: 
\[
\rho_1^h(y) = \int  \frac{\rho_0(x)}{\tLambda_h(x)} \exp\left( -\frac{1}{h} D[y \mid x^*]\right) dx.
\]
Note that $\support(\rho_1^h)=\support(\rho_1)$, since the divergence is infinity outside this set. 

The map $x \mapsto x^*$ is invertible on $\support_0(\rho_1)$ and the inverse is measurable. Change the variable to $z=x^*$ and denote inverse as $x=z_*$. Then
\eq\label{eq:taylorrho1}
\rho_1^h(y)  = \int  \frac{\rho_1(z)}{\tLambda_h(z_*)}  \exp\left( -\frac{1}{h}D[y \mid z] \right) dz.   
\en

Note that
\eq\label{eq:ptwisediff}
\rho_1^h(y) - \rho_1(y) = \int  \frac{\rho_1(z) - \rho_1(y)}{\tLambda_h(z_*)}  \exp\left( -\frac{1}{h}D[y \mid z] \right) dz.  
\en
Let $y\in \support_0(\rho_1)$ be such that $A(y)$ exists and the conditions in Assumption \ref{asmp:continuity} are satisfied. As before, fix a $\delta >0$ such that $B_\delta(y) \subseteq \support(\rho_1)$. Split the integral inside $B_\delta(y)$ and outside. Inside $B_\delta(y)$, we do a change of measure, 
\[
\begin{split}
\frac{1}{\tLambda_h(z_*)} e^{-\frac{1}{h} D[y \mid z]} &= \frac{(\sqrt{2\pi h})^d}{\Lambda_h(z_*) \sqrt{\abs{J}(y)}} e^{-\frac{1}{h} D[y \mid z] + \frac{1}{2h} (z-y)^T A(y) (z-y)}\\
&\times \frac{\sqrt{\abs{J}(y)}}{(\sqrt{2\pi h})^d} e^{-\frac{1}{2h} (z-y)^T A(y) (z-y)}. 
\end{split}
\]

Note that, for almost every $z\in B_\delta(y)$, the following is well-defined,  
\[
-\frac{1}{h}D[y \mid z] + \frac{1}{2h} (z-y)^T A(y) (z-y)= \frac{1}{2h} (z-y)^T \left( A(y) - A(z) \right) (z-y) - \frac{1}{h}R[y \mid z].
\]
By a.s. continuity of the matrix valued map $A(\cdot)$, the density $\rho_1$, and assumption \eqref{eq:remain1}, for any $\epsilon >0$, we can choose $\delta >0$ such that 
\[
\begin{split}
&\abs{\frac{1}{2h} (z-y)^T \left( A(y) - A(z) \right) (z-y) -\frac{1}{h} R[ y \mid z]} \le \frac{\epsilon}{2h} (z-y)^T A(y) (z-y) , \; \forall z\in B_\delta(y),\\
& \text{and}\quad 1- \epsilon < \frac{ \tLambda_h(y_*)}{ \tLambda_h(z_*)} \le 1 + \epsilon.
\end{split}
\]

For such a choice of $\delta$, 
\[
\frac{(\sqrt{2\pi h})^d}{\tLambda_h(z_*) \sqrt{\abs{J}(y)}} e^{-\frac{1}{h} D[y \mid z] + \frac{1}{2h} (z-y)^T A(y) (z-y)}\le \frac{(1+\epsilon) (\sqrt{2\pi h})^d}{\tLambda_h(y_*) \sqrt{\abs{J}(y)}} e^{-\frac{\epsilon}{2h} (z-y)^T A(y) (z-y)}.
\]
Hence
\[
\frac{1}{\tLambda_h(z_*)} e^{-\frac{1}{h} D[y \mid z]} \le (1+\epsilon)\frac{(\sqrt{2\pi h})^d}{\tLambda_h(y_*) \sqrt{\abs{J}(y)}}  \times  \frac{\sqrt{\abs{J}(y)}}{(\sqrt{2\pi h})^d} e^{-\frac{(1-\epsilon)}{2h} (z-y)^T A(y) (z-y)}.
\]

Now use Lemma \ref{lem:partitionfnasymp} (ii) to choose $h$ small enough such that $(\sqrt{2\pi h})^d \le (1+\epsilon) \tLambda_h(y_*) \sqrt{\abs{J}(y)}$. Thus, for all sufficiently small $h$, 
\[
\begin{split}
\int_{B_\delta(y)} &\abs{\rho_1(z)-\rho_1(y)} \frac{1}{\Lambda_h(z_*)} e^{-\frac{1}{h} D[y \mid z]}dz \\
\le & (1+\epsilon)^2 \int_{B_\delta(y)}  \abs{\rho_1(z)-\rho_1(y)} \frac{\sqrt{\abs{J}(y)}}{(\sqrt{2\pi h})^d} e^{-\frac{(1-\epsilon)}{2h} (z-y)^T A(y) (z-y)}\\
\le &(1+\epsilon)^2\int_{\rr^d}  \abs{\rho_1(z)-\rho_1(y)} \frac{\sqrt{\abs{J}(y)}}{(\sqrt{2\pi h})^d} e^{-\frac{(1-\epsilon)}{2h} (z-y)^T A(y) (z-y)}.
\end{split}
\]
By the assumed continuity of $\rho_1$, the above converges to zero as we take $\delta \downarrow 0$.

Now for the integral outside $B_\delta(y)$ we show it is negligible as in the proofs of parts (i) and (ii) of Lemma \ref{lem:partitionfnasymp}.
As in there, there exists constants $m, m'>0$ such that
\[
\int_{B^c_\delta(y)} \abs{\rho_1(z)-\rho_1(y)}  \frac{1}{\Lambda_h(z_*)}e^{-\frac{1}{h} D[y \mid z]} \le m' h^{-d/2} e^{-m/h}. 
\] 
From \eqref{eq:ptwisediff}, the pointwise convergence now follows.

Moreover, using the cone condition in Assumption \ref{asmp:density} and the fact that $\rho_1$is bounded away from zero, one can prove as before, uniform upper and lower bounds, i.e., for some positive constant $C$,
\eq\label{eq:densityratiobound}
\frac{1}{C} \le \frac{\rho_1^h(y)}{\rho_1(y)} \le C,\quad \text{for all $y \in \support_0(\rho_1)$}.
\en
We skip the similar details.
\end{proof}

\begin{proof}[Proof of the Theorem \ref{thm:mainthmw}] It is enough to prove the claimed limit of $K_h$. Since, for $K_h'$ use the relationship from \eqref{eq:khkhp}
\[
K'_h = K_h + \Ent(\rho_0) - \int \log \Lambda_h(x) \rho_0(x) dx.
\] 
By Lemma \ref{lem:partitionfnasymp} (i),
\[
\lim_{h\rightarrow 0+}\left[ \int \log \Lambda_h(x) \rho_0(x) dx - \frac{d}{2}\log(2\pi h)\right]= -\frac{1}{2} \log \det \nabla^2 g(0).
\]  
Thus, asymptotically, 
\[
K'_h= \Ent(\rho_0) - \frac{d}{2} \log(2\pi h) + \frac{1}{h}\wass_g(\rho_0, \rho_1)+ H(\rho_1 \mid \mu_g) + o(1).
\]
This gives us the statement regarding $K_h'$.

The limit of $K_h$ is split separately as a lower and an upper bound. The lower bound is easy. Consider the exponential tilting $\tmu_h$ from Lemma \ref{lem:partitionfnasymp} (iv). Recall that $\mu_h$ denotes the joint density $\rho_0(x)p_h(x,y).$
Then, from \eqref{eq:transitiontilt}, 
\[
\frac{d\tmu_h}{d\mu_h}(x,y)=\frac{\tp_h(x,y)}{p_h(x,y)}=\frac{1}{Z_h(x)} \exp\left( \frac{1}{h}\left( \psi(x) + \psi^*(y) \right)\right)
\]
For any $\nu\in M_1(\Sigma)$, absolutely continuous with respect to $\tmu_h$,
\[
\begin{split}
H(\nu \mid \tmu_h)&= \nu\left( \log \frac{d\nu}{d\tmu_h}\right)=\nu\left( \log \frac{d\nu}{d\mu_h} - \log \frac{d\tmu_h}{d\mu_h} \right)\\
&= H\left( \nu \mid \mu_h \right) - \frac{1}{h}\nu\left( \psi(x) + \psi^*(y)\right) + \nu\left( \log Z_h(x)\right).
\end{split}
\]

If moreover $\nu \in \Pi(\rho_0, \rho_1)$, then, by Kantorovich duality, 
\[
\begin{split}
\frac{1}{h} \int \left( \psi(x) + \psi^*(y)\right) d\nu= \frac{1}{h}\left[ \int \psi(x) \rho_0(x) dx + \int \psi^*(y) \rho_1(y) dy\right]= \frac{1}{h} \wass_g(\rho_0, \rho_1)
\end{split}
\]
and $\nu\left( \log Z_h(x)\right) = \int \log\left( Z_h(x)\right) \rho_0(x) dx < \infty$ by Lemma \ref{lem:partitionfnasymp} (iii). 

Thus, alternatively, $K_h$ from \eqref{eq:entropic} is given by
\[
K_h = \frac{1}{h}\wass_g(\rho_0, \rho_1) - \int \log\left( Z_h(x)\right) \rho_0(x) dx + \inf_{\nu \in \Pi(\rho_0, \rho_1)} H\left( \nu \mid \tmu_h \right).
\]
Since $ H\left( \nu \mid \tmu_h \right)\ge 0$, we immediately get the lower bound
\eq\label{eq:lbnd}
K_h - \frac{1}{h}\wass_g(\rho_0, \rho_1) \ge - \int \log\left( Z_h(x)\right) \rho_0(x) dx = -\frac{1}{2} \log \det \nabla^2 g(0) + H(\rho_1 \mid \mu_g) + o(1),
\en
as $h\rightarrow 0+$ (by Lemma \ref{lem:partitionfnasymp}) (iii).

We now obtain a matching upper bound by constructing a coupling of $(\rho_0, \rho_1)$ that is close to $\tmu_h$ in relative entropy. Let $\rho_1^h$ be the density of $Y$ under $\tmu_h$. Recall that the measurable inverse Monge map $x \mapsto x^*$ exists on $\support_0(\rho_1)$. We will denote this inverse by the notation $z \mapsto z_*$. Push-forward the joint density $\tmu_h(x,y)$ by the map $x\mapsto z=x^*$ to obtain
\[
\tnu_h(z, y):= \rho_1(z) \frac{1}{\tLambda_h(z_*)} e^{- \frac{1}{h} D[y \mid z]}.
\]
Under $\tnu_h$, the density of $Z$ is $\rho_1$ and the density of $Y$ is $\rho_1^h$. Consider the conditional density of $Z$, given $Y=y$, 
\[
\tchi_h(y, z)= \frac{\rho_1(z)}{\rho_h^1(y)}\frac{1}{\tLambda_h(z_*)} e^{- \frac{1}{h} D[y \mid x^*]},
\]
and the joint density of $(X, Y, Z)$ obtained by $\rho_0(x) \tp_h(x, y) \tchi_h(y, z)$. This represents a two-step Markov chain where $X$ has density $\rho_0$, $Y$, given $X=x$, has density $\tp_h(x,\cdot)$, and $Z$, given $X=x, Y=y$, has density $\tchi_h(y, \cdot)$. Integrate out $Y$ from the above joint density to obtain
\[
f_{2h}(x, z):= \int \rho_0(x) \tp_h(x, y) \tchi_h(y, z) dy = \frac{\rho_0(x)}{\tLambda_h(x) \tLambda_h(z_*)} \int \frac{\rho_1(z)}{ \rho_1^h(y)} e^{-\frac{1}{h} D[y \mid x^*] - \frac{1}{h} D[y \mid z]} dy.
\]

It is obvious from our construction that $f_{2h}$ is a coupling of $(\rho_0, \rho_1)$. We will show that 
\eq\label{eq:2hre}
\lim_{h \rightarrow 0+} H\left( f_{2h} \mid \tmu_{2h} \right)=0. 
\en
This will prove our upper bound (by replacing $h$ by $h/2$) and complete the proof of the theorem. The rest of the proof is devoted to proving \eqref{eq:2hre}.

Before we prove the claim let us give an outline of the argument. For $h \approx 0$, under $ (X, Z) \sim \tp_{2h}$, $X \sim \rho_0$ and $Z$, given $X=x$, is approximately $N(x^*, 2h A^{-1}(x^*))$. Similarly, consider $(X,Y,Z) \sim \rho_0(x) \tp_h(x, y) \tchi_h(y, z)$. Then, given $X=x$, $Y \approx N(x^*, h A^{-1}(x^*))$. What we show is that $Z$, given $Y=y, X=x$, is also approximately $N\left( y, h A^{-1}(x^*)\right)$. Hence, $Z-x^*$, given $X=x$, is approximately the sum of two i.i.d. $N(0, h A^{-1}(x^*))$ random variables, whose distribution is $N(0, 2h A^{-1}(x^*))$. Hence, $f_{2h} \approx \tmu_{2h}$ in the sense of relative entropy.   

To proceed with the proof, fix an arbitrary $\epsilon >0$. We will show 
\eq\label{eq:limsupinf}
\limsup_{h \rightarrow 0+} H\left( f_{2h} \mid \tmu_{2h} \right) \le \epsilon.
\en
Since $H\left( f_{2h} \mid \tmu_{2h} \right)\ge 0$, by taking $\epsilon$ to zero we get our result. 

Consider $f(x,z)$. Recall that, for any $\epsilon >0$, there exists $\delta >0$ such that \eqref{eq:remain1} holds. For such a $\delta$, consider two possible cases: either $\norm{x^* - z} \le 2\delta$ or $\norm{x^* - z} > 2\delta$. In the latter case, for every $y \in \rr^d$, it belongs to either $A:=\{y :\; \norm{y - x^*} > \delta\}$ or $B:=\{ y:\; \norm{y - z} > \delta\}$. Thus, by \eqref{eq:divsublev}, there is an $m>0$, such that either $D[y \mid x^*] > m$ or $D[y \mid z] > m$. Since $D[y \mid \cdot]=\infty$ if $y \notin \support(\rho_1)$, if $\norm{x^* - z} \ge 2\delta$, 
\[
\begin{split}
f_{2h}(x, z) &:\le \frac{\rho_0(x)\rho_1(z)}{\tLambda_h(x) \tLambda_h(z_*)} e^{-m/h}\int_{AB^c \cap \support(\rho_1)} \frac{1}{ \rho_1^h(y)} e^{- \frac{1}{h} D[y \mid z]} dy\\
&+ \frac{\rho_0(x)\rho_1(z)}{\tLambda_h(x) \tLambda_h(z_*)} e^{-m/h} \int_{A^cB\cap \support(\rho_1)} \frac{1}{ \rho_1^h(y)} e^{-\frac{1}{h} D[y \mid x^*]} dy\\
&+ \frac{\rho_0(x)\rho_1(z)}{\tLambda_h(x) \tLambda_h(z_*)} e^{-2m/h} \int_{AB\cap \support(\rho_1)} \frac{1}{ \rho_1^h(y)} dy.
\end{split}
\]
By known asymptotics of $\tLambda$ and the pointwise limit of $\rho_1^h$ proved in Lemma \ref{lem:partitionfnasymp}, for all small enough $h$, the above is bounded by $C \sqrt{h} e^{-m/h}$, for a.e. $(x,z)\in \support(\rho_0) \times \support(\rho_1)$.  

Now, consider the case when $\norm{x^* - z}< 2\delta$. In the integral involving $y$, by the previous argument, we only need to consider the case when $\max\left(\norm{y -x^*}, \norm{y-z} \right)< \delta$ since the rest is again exponentially small. In that domain, we use \eqref{eq:remain1} to replace the divergence by its quadratic approximation. Also in that domain, by choosing a smaller $\delta$, if necessary, one can guarantee that for all small enough $h$, 
\[
\frac{\rho_1(z)}{\rho_1^h(y)}=\frac{\rho_1(y)}{\rho_1^h(y)} \frac{\rho_1(z)}{\rho_1(y)} < (1+\epsilon), \quad \text{and}\quad \norm{A^{-1}(x^*)A(z) - I} \le \epsilon. 
\]
In particular, 
\[
\begin{split}
D[y \mid x^*] &+ D[y \mid z] \ge \frac{1}{2} (y-x^*)^T A(x^*) (y -x^*) +  \frac{1}{2} (y-z)^T A(z) (y -z) \\
&- \epsilon \norm{y-z}^2 - \epsilon \norm{y-x^*}^2\\
&\ge \frac{1}{2} (y-x^*)^T A(x^*) (y -x^*) +  \frac{1}{2} (y-z)^T A(x^*) (y -z) - 3\epsilon \delta^2.
\end{split}
\]

Hence, 
\[
\begin{split}
&f_{2h}(x, z) - C \sqrt{h} e^{-m/h}\le \frac{\rho_0(x)}{ \tLambda_h(x) \tLambda_h(z_*)} (1+\epsilon) \int_{B_\delta(x^*) \cap B_\delta(z)}  e^{- \frac{1}{h} D[y \mid x^*] - \frac{1}{h} D[y \mid z]} dy \\
&\le (1+\epsilon) \frac{\rho_0(x)}{ \tLambda_h(x) \tLambda_h(z_*)} e^{-3\epsilon \delta^2} \int_{B_\delta(x^*) \cap B_\delta(z)} e^{- \frac{1}{2h} (y-x^*)^T A(x^*) (y -x^*) -  \frac{1}{2h} (y-z)^T A(x^*) (y -z)} dy \\
&\le (1+\epsilon) e^{-3\epsilon \delta^2} \frac{\rho_0(x)}{ \tLambda_h(x) \tLambda_h(z_*)} e^{-3\epsilon \delta^2} \int_{\rr^d} e^{- \frac{1}{2h} (y-x^*)^T A(x^*) (y -x^*) -  \frac{1}{2h} (y-z)^T A(x^*) (y -z)} dy. 
\end{split}
\]
The integral inside is a Gaussian integral and is best dealt best by the probability argument outlined above. Suppose $(Y', Z')$ is jointly Gaussian where $Y'\sim N(x^*, hA^{-1}(x^*))$ and $Z'$, given $Y'=y$, is $N(y, hA^{-1}(x^*))$. Then, the distribution of $Z$ is $N(x^*, 2hA^{-1}(x^*))$. 

Hence, 
\[
\begin{split}
\frac{ \det A(x^*)}{(2\pi h)^d} \int_{\rr^d} & e^{- \frac{1}{2h} (y-x^*)^T A(x^*) (y -x^*) -  \frac{1}{2h} (y-z)^T A(x^*) (y -z)} dy\\
&= \frac{\det A^{1/2}(x^*)}{(4\pi h)^{d/2}} e^{-\frac{1}{4h} (z-x^*)^T A(x^*) (z-x^*)}. 
\end{split}
\]

Therefore, 
\[
\begin{split}
&f_{2h}(x, z) - C \sqrt{h} e^{-m/h}\le (1+\epsilon) e^{-3\epsilon \delta^2} \frac{\rho_0(x)}{ \tLambda_h(x) \tLambda_h(z_*)} e^{-3\epsilon \delta^2}\\
&\times \frac{(2\pi h)^d} { \det A(x^*)}\frac{\det A^{1/2}(x^*)}{(4\pi h)^{d/2}} e^{-\frac{1}{4h} (z-x^*)^T A(x^*) (z-x^*)}. 
\end{split}
\]
Now, for all small enough $h$, we can (i) absorb the $O(\sqrt{h} e^{-m/h})$ term with an extra $O(\epsilon)$ term, (ii) lower bound both $\tLambda_h(x)$ and $\tLambda_h(z_*) \det A^{1/2}(x^*)$ by $(1-\epsilon)(\sqrt{2\pi h})^d$ to get 
\[
\begin{split}
f_{2h}(x, z) &\le (1+ 2\epsilon) (1-\epsilon) e^{-3\epsilon \delta^2} \rho_0(x) \frac{1}{ \tLambda_h(x)} 2^{-d/2} e^{-\frac{1}{4h} (z-x^*)^T A(x^*) (z-x^*)}\\
&\le (1+ 2\epsilon) (1-\epsilon)^2 e^{-3\epsilon \delta^2} \rho_0(x) \frac{\det A^{1/2} (x^*)}{(4\pi h)^{d/2}} e^{-\frac{1}{4h} (z-x^*)^T A(x^*) (z-x^*)}\\
&\le (1+ 3\epsilon) \rho_0(x) \frac{\det A^{1/2} (x^*)}{(4\pi h)^{d/2}} e^{-\frac{1}{4h} (z-x^*)^T A(x^*) (z-x^*)},
\end{split}
\]
by choosing $\epsilon, h$ small enough. 

Hence, for a.e. $x\in \support(\rho_0)$ and $z\in \support(\rho_1)$, we get the estimate
\eq\label{eq:f2hest}
f_{2h}(x, z) \le\begin{cases}
(1+ 3\epsilon) \rho_0(x) \tq_{2h}(x, z), & \norm{x-z^*} \le 2\delta,\\
C \sqrt{h} \rho_0(x) e^{-m/h}, & \norm{x-z^*} > 2\delta. 
\end{cases}
\en
In particular, from \eqref{eq:remain1},
\[
\frac{f_{2h}(x,z)}{\tmu_{2h}(x,z)} \le \begin{cases}
(1+3\epsilon) \frac{\det A^{1/2} (x^*)\tLambda_{2h}(x)}{(4\pi h)^{d/2} } e^{\frac{1}{2h} R[z \mid x^*]}, & \norm{x-z^*} \le 2\delta,\\
C \sqrt{h} \tLambda_{2h}(x) e^{\frac{1}{2h} D[z \mid x^*] - \frac{m}{h}}, & \norm{x-z^*} > 2\delta.
\end{cases}
\]
For all small enough $h$, $(1+3\epsilon) \frac{\det A^{1/2} (x^*)\tLambda_{2h}(x)}{(4\pi h)^{d/2} } \le (1+4\epsilon) \le e^{4\epsilon}$.
Moreover, $D[z \mid x^*]$ is bounded when $z\in \support(\rho_1)$ and $x\in \support(\rho_0)$.
Thus,
\[
\begin{split}
H\left( f_{2h} \mid \tmu_{2h} \right)&= \int \log \left( \frac{f_{2h}(x,z)}{\tmu_{2h}(x,z)} \right) dzdx\\
&\le 4\epsilon + \frac{1}{2h}\int_{\norm{x-z^*} < 2\delta} R[z \mid x^*] f_{2h}(x,z) dzdx \\
&+ C' \frac{\log h}{h} \int_{\norm{x-z^*} \ge 2\delta} f_{2h}(x,z)dzdx.
\end{split}
\]
Using the fact that $R[z \mid x^*] \le \epsilon \norm{z-x^*}^2$ and the estimate \eqref{eq:f2hest}, we get 
\[
H\left( f_{2h} \mid \tmu_{2h} \right) \le 4 \epsilon + O(\epsilon) + C' \frac{\log h}{h} e^{-m/h}. 
\] 
Hence, $\limsup_{h\rightarrow 0+}H\left( f_{2h} \mid \tmu_{2h} \right)=O(\epsilon)$ showing \eqref{eq:limsupinf}. 
\end{proof}

\section{Dirichlet transport and the proof of Theorem \ref{thm:mainthm2}}\label{sec:Dirichlet} 

It is clear that Theorem \ref{thm:mainthmw} can be generalized to cost functions satisfying the Twist Condition for which $c(x, \cdot)$ is locally quadratic at $x$. However, one drawback of that result is that it is coordinate dependent. That is, if we consider a diffeomorphism $\mathbb{S}:\rr^d \rightarrow \rr^d$ and change the cost function to $c\left( \mathbb{S}(x), \mathbb{S}(y)\right)$, then we would get different limits in Theorem \ref{thm:mainthmw}. This section, however, shows that (i) some diffeomorphisms are more natural for the problem, and (ii) gives us strikingly clean limit functions.

Recall the cost function on the unit simplex $\Delta_n$ given in \eqref{eq:costDirichlet}:
\[
c(p,q)= \log\left( \frac{1}{n} \sum_{i=1}^n \frac{q_i}{p_i} \right) -\frac{1}{n} \sum_{i=1}^n \log \frac{q_i}{p_i}, \quad p,q \in \Delta_n. 
\]

The so-called exponential coordinate system of the unit simplex refers to the following map from $\Delta_n \rightarrow \rr^{n-1}$:
\[
\mathbf{\Theta}(p)=\log\left(p_i/p_n \right), \quad i=1,2,\ldots, n-1. 
\]
The inverse of this map takes $(\theta_1, \ldots, \theta_{n-1} )\in \rr^{n-1}$ to an element $p \in \Delta_n$, where
\[
p_i= \frac{e^{\theta_i}}{1+ \sum_{j=1}^{n-1} e^{\theta_j}},\; i=1,2,\ldots, n-1, \quad p_n=\frac{1}{1+\sum_{j=1}^{n-1} e^{\theta_j}}.
\]
If $\mathbf{\Theta}(p)=\theta=(\theta_1, \ldots, \theta_{n-1})$ and $\mathbf{\Theta}(q)=\rho=(\rho_1, \ldots, \rho_{n-1})$ are the respective exponential coordinates of $p$ and $q$, one can express $c(p,q)$ as
\[
\widetilde{c}(\theta, \rho)= \log\left(\frac{1}{n} + \frac{1}{n} \sum_{i=1}^{n-1} e^{\rho_i- \theta_i}\right) - \frac{1}{n} \sum_{i=1}^{n-1} (\rho_i-\theta_i)=g(\theta-\rho),
\]
where $g$ is indeed a strictly convex cost function. For a detailed analysis using this coordinate system see \cite{PW16}. 

Thus, Theorem \ref{thm:mainthmw} covers the limit for this cost function. However, the exponential coordinate system is not the most natural coordinate system for this transport problem and the limit we get from applying Theorem \ref{thm:mainthmw} is not illuminating. 
The multiplicative group $\odot$ appears to give it more natural properties, especially related to the behavior of entropy. We start by describing the Monge solutions.


A function $\varphi:\Delta_n \rightarrow [-\infty, \infty)$ is called exponentially concave if $e^{\varphi}$ is a (nonnegative) concave function on $\Delta_n$. Hence $\varphi$ has directional derivatives everywhere it is finite, and, indeed, is differentiable Lebesgue almost everywhere on that domain. 

\begin{defn}[Portfolio map]
Let $\varphi$ be exponentially concave on $\Delta_n$. Define $\boldsymbol{\pi}(r) \in \overline{\Delta}_n$ by
\begin{equation} \label{eqn:portfolio.map}
(\boldsymbol{\pi}(r))_i = r_i \left(1 + \nabla_{e_i - r} \varphi(r)\right), \quad i = 1, \ldots, n,
\end{equation}
where $\{e_1, \ldots, e_n\}$ is the standard basis of $\mathbb{R}^n$ and $\nabla_{e_i - r}$ is the directional derivative in the direction $e_i-r$. We call $\boldsymbol{\pi}$ the portfolio map generated by $\varphi$.
\end{defn}

The fact that the range of $\boldsymbol{\pi}$ is contained in $\Delta_n$ has been first observed by Fernholz in the context of stochastic portfolio theory (hence the name). See \cite{F02, PW16}. The map $\boldsymbol{\pi}(r)$ is the multiplicative equivalent of the additive map $x\mapsto x - \nabla \phi(x)$, for a convex function $\phi$, for the quadratic Wasserstein transport problem.

Consider two probability densities $\rho_0, \rho_1$ that are compactly supported in $\Delta_n$. Consider the optimal transport from $\rho_0$ to $\rho_1$ with cost $c$. There exists a unique solution to the Monge problem whose solution can be described in terms of the portfolio map of an exponentially concave function. The following is taken from \cite[Theorem 4 and eqn. (38)]{PW18}.

There exists an exponentially concave function $\varphi: \Delta_n \rightarrow \mathbb{R}$ such that the following statements hold.
\begin{itemize}
\item[(i)] If $\boldsymbol{\pi}$ is the portfolio map generated by $\varphi$, the mapping $p \mapsto T(p)=q$ where
\begin{equation} \label{eqn:deterministic.transport}
q\odot p^{-1} := \boldsymbol{\pi}(p^{-1}),
\end{equation}
is $\rho_0$ a.e.~defined and pushforwards $\rho_0$ to $\rho_1$.
\item[(ii)] The deterministic coupling $(p,T(p))$ is the optimal Monge solution and is $\rho_0$-a.e.~unique.
\end{itemize}

However, we are more interested in the dual transport map from $\rho_1$ to $\rho_0$. To derive this easily, note that $c(p,q)=c(q^{-1}, p^{-1})$ from straightforward algebra. Let $\rho_0^-$ (respectively, $\rho_1^-$) denote the push-forward of $\rho_0$ (respectively, $\rho_1$) by the map $p \mapsto p^{-1}$. Consider the transport with cost $c$ from $\rho_1^-$ to $\rho_0^-$. By the previous assertions, there is an exponentially concave function $\varphi$ that gives the Monge solution. Let $T$ and $\boldsymbol{\pi}$ denote the corresponding Monge map and the portfolio map. Replacing $p$ to $q^{-1}$ and $q$ to $p^{-1}$, we see that there exists a map $q\mapsto p= q_*=\left(T(q^{-1})\right)^{-1}$ that transports $\rho_1$ to $\rho_0$ and such that $q\odot p^{-1} = \boldsymbol{\pi}(q)$. We will call this map $\bpi$ the dual portfolio map. To drive home the similarity with the notation in the previous section, let $p \mapsto p^*=\left(T^{-1}(p^{-1})\right)^-$ denote the inverse map.

By Alexandrov's theorem, $\varphi$ is twice differentiable almost everywhere on its domain. By exponential concavity, the following matrix  then becomes nonpositive definite: 
\eq\label{eq:ldivmatrix}
\tnabla^2 \varphi(p) = \nabla^2 \varphi(p) + \nabla \varphi(p) \nabla^T\varphi(p) \le 0. 
\en

The divergence for an exponentially concave function is called the $L$-divergence and was originally defined in \cite{PW14}. More detailed studies on the geometry of $L$-divergence can be found in \cite{W15, PW16, W17}. The following definition is equivalent to Definition \ref{defn:cdiv} by passing to exponential coordinates (see \cite{PW16}) but avoids developing a duality theory for exponentially concave functions. 

\begin{defn} [$L$-divergence] \label{def:L.divergence}
Let $\varphi$ be a differentiable exponentially concave function on $\Delta_n$. The $L$-divergence of $\varphi$ is a function from $\Delta_n \times \Delta_n$ to $[0, \infty]$ given by the following expression. For any two $q, r$, let
\begin{equation} \label{eqn:L.divergence}
{\bf L}\left[ r \mid q \right] =  \log \left( \sum_{i=1}^n \bpi_i(q) \frac{r_i}{q_i} \right) - \left( \varphi(r) - \varphi(q)\right).
\end{equation}
\end{defn}

By the exponential concavity of $\varphi$, it can be shown that ${\bf L}\left[ r \mid r' \right] \geq 0$ and ${\bf L}\left[ r \mid r \right] = 0$. If $e^{\varphi}$ is strictly concave, then ${\bf L}\left[ r \mid r' \right] > 0$ for all $r \neq r'$. As $r \rightarrow q$, we can now consider a Taylor approximation to $L$-divergence. More details about the following calculations can be found in \cite[Section 4.2]{PW18}. Let ${L}(q): = -\wt{\nabla}^2 \varphi(q)$ from \eqref{eq:ldivmatrix} denote a positive semidefinite matrix. Then, under appropriate smoothness, ${\bf L}\left[ r \mid q \right] \approx \frac{1}{2} (r-q)^T L(q) (r-q)$. 

We make the following assumptions that mirror Assumptions \ref{asmp:density} and \ref{asmp:continuity}. 

\begin{asmp}\label{asmp:densityc} Assume that $\rho_0, \rho_1$ are continuous and compactly supported in $\Delta_n$ and $\rho_1$ is bounded away from zero on its support. Further assume that there is angle $\theta > 0$ and a constant $\delta > 0$ such that, for any $y \in \support(\rho_1)$, there exists $w \neq 0$ (depending on $y$), such that the shifted cone $y + K(w, \theta)$ intersection with $B_\delta(y)$ lies completely inside $\support(\rho_1)$. Finally, assume that the Lebesgue measure of the boundary of $\support(\rho_1)$ is $0$.
\end{asmp}

\begin{asmp}\label{asmp:continuityc}
Assume that for any $\delta >0$, 
\eq\label{eq:divsublev}
\inf_{z\in \support(\rho_1)} \inf_{y \in B^c_\delta(z)} \ldiv[y \mid z] >0,
\en   
where $B^c_\delta(z)$ is the complement of the open Euclidean ball around $z$ of radius $\delta$. Assume also that the following quadratic approximation to the divergence holds uniformly on $\support_0(\rho_1)$. For any $\epsilon >0$, there is a $\delta>0$ such that for all $z \in \support_0(\rho_1)$ and all $y\in \support(\rho_1)$ such that $\norm{y-z}< \delta$, we have  
\eq\label{eq:ldivquad}
R[r\mid r'] := \ldiv[r\mid r']- \frac{1}{2}({r}-{r'})^T L(r') ({r}-{r'}) = o\left( \abs{{r}-{r'}}^2\right), \quad \text{as} \; r\rightarrow r',
\en
for some family of matrices $L(\cdot)$ that is continuous a.s. on $\support_0(\rho_1)$. The family $L(\cdot)$ is assumed to be uniformly elliptic in the sense that there exists an $\varepsilon>0$ such that the eigenvalues of $L(r)$ lie in $(\varepsilon, 1/\varepsilon)$ for all $r \in \support_0(\rho_1)$. In particular, the family of inverses $L^{-1}(z)$, $z\in \support_0(\rho_1)$, exists and is continuous almost surely.
\end{asmp}

Now, suppose $q=p^*$ is given by a Monge map with the dual portfolio map $\pi=\bpi(q)=q\odot p^{-1}$. Then, for any other $r \in \Delta_n$, 
\[
\log\left( \frac{1}{n} \sum_{i=1}^n \frac{r_i}{p_i}\right) - \log\left( \frac{1}{n} \sum_{i=1}^n \frac{q_i}{p_i}\right)= \log\left( \sum_{i=1}^n \pi_i \frac{r_i}{q_i} \right).
\]
Hence, assuming that $\boldsymbol{\pi}$ is the dual portfolio map, if $q=p^*$ we can rewrite \eqref{eqn:L.divergence} in the alternative form
\begin{equation}  \label{eqn:L.divergence.portfolio}
\begin{split}
{\bf L}\left[ r \mid p^* \right] &= \log\left( \frac{1}{n}\sum_{i=1}^n \frac{r_i}{p_i} \right) - \log \left( \frac{1}{n}\sum_{i=1}^n \frac{p^*_i}{p_i} \right)  - \left( \varphi(r) - \varphi(p^*)\right)\\
&= c(p,r) - c(p, p^*) + \frac{1}{n} \sum_{i=1}^n \log \frac{r_i}{p_i^*}  - \left( \varphi(r) - \varphi(p^*)\right).
\end{split}
\end{equation}
Hence, 
\eq\label{eq:feepq}
\begin{split}
c(p,r) - \ldiv[r \mid p^*] = c(p, p^*) +  \frac{1}{n} \sum_{i=1}^n \log p^*_i - \frac{1}{n} \sum_{i=1}^n \log r_i + \varphi(q) - \varphi(p^*).
\end{split}
\en

\subsection{The Dirichlet transport} The global coordinate system we use on $\Delta_n$ is given by the first $(n-1)$ coordinates of $p\in \Delta_n$. It will be convenient to have a new notation:
\[
\tp = (p_1, \ldots, p_{n-1}),\quad \text{or},\; p_i= \tp_i,\; \text{for $i=1,2,\ldots, n-1$ and}\; p_n = 1 - \sum_{i=1}^{n-1} \tp_i. 
\]
In the calculations below, all integrals are performed with respect to the Lebesgue measure on the set 
\[
\left\{\tp=(p_1, \ldots, p_{n-1}),\; p_i > 0, \forall \; i,\; \sum_{i=1}^{n-1} p_i < 1 \right\}.
\]

Recall the symmetric Dirichlet distribution $\Diri(\lambda)$ from the Introduction \eqref{eq:diridensity} and the density $\nu_0$ of $\Diri(0)$. Fix $p \in \Delta_n$, generate $G \sim \Diri(\lambda)$ and consider the random variable $p \odot G$. The density of the random variable (see \cite{PW18}) at any $q\in \Delta_n$ is 
\[
\begin{split}
F_\lambda(p, q) &= \frac{\Gamma(\lambda)}{\left(\Gamma(\lambda/n)\right)^n} \frac{1}{\prod_{i=1}^n p_i} \prod_{i=1}^n \left( \frac{q_i}{p_i}\right)^{\frac{\lambda}{n}-1} \left( \sum_{i=1}^n \frac{q_i}{p_i}\right)^{-\lambda}\\
&= \frac{\Gamma(\lambda)n^{-\lambda}}{\left(\Gamma(\lambda/n)\right)^n} \nu_0(q) \exp\left( -\lambda c(p,q)\right).
\end{split}
\]

We start by carefully analyzing the normalizing constant in the above density as $\lambda \rightarrow \infty$. By Stirling approximation to the gamma function
\eq\label{eq:stirling}
\begin{split}
& \frac{\Gamma(\lambda)n^{-\lambda}}{\left(\Gamma(\lambda/n)\right)^n} \sim  \sqrt{\frac{2\pi}{ \lambda}} \left( \frac{\lambda}{e}\right)^\lambda \frac{1}{n^\lambda} \left(\sqrt{\frac{\lambda}{2\pi n}}\right)^n \left( \frac{n e}{\lambda}\right)^\lambda \\
& = (2\pi)^{-(n-1)/2} \lambda^{(n-1)/2} n^{-n/2} = n^{-n/2}\left(\frac{\lambda}{2\pi}\right)^{(n-1)/2}.
\end{split}
\en
Here $\sim$ means that the ratio of the two terms converges to one as $\lambda \rightarrow \infty$. 

This gives us a transition probability on $\Delta_n$ such that $\lim_{\lambda \rightarrow \infty} -\frac{1}{\lambda} \log F_\lambda(q\mid p) = c(p,q)$. Thus $h:=\frac{1}{\lambda}$ is a measure of noise, and as $h \rightarrow 0+$, we recover the optimal transport with cost $c$. 
In particular, from \eqref{eq:stirling} notice that
\[
F_\lambda(p,q) = \frac{\Gamma(\lambda)n^{-\lambda}}{\left(\Gamma(\lambda/n)\right)^n} \nu_0(q) e^{-\lambda c(p,q)} \sim \frac{n^{-n/2} }{(2\pi h)^{(n-1)/2}} \nu_0(q) e^{-\frac{1}{h} c(p,q)}, \quad \text{for large $\lambda$},
\]

The proof of Theorem \ref{thm:mainthm2} is very similar to that Theorem \ref{thm:mainthmw}. We will reuse similar notations to stress this point and skip similar steps in the proof. First we replace the transition density $F_\lambda$ by 
\eq\label{eq:stirlingfh}
f_h(p,q)= F_{1/h}(p,q) = \frac{\Gamma(1/h)n^{-1/h}}{\left(\Gamma(1/nh)\right)^n} \nu_0(q) e^{-\frac{1}{h} c(p,q)} \sim \frac{n^{-n/2}}{\left({2\pi h}\right)^{(n-1)/2}}\nu_0(q) e^{-\frac{1}{h} c(p,q)}
\en 
This gives us a joint density $\mu_h(p,q)=\rho_0(p)f_h(p,q)$.

We will now define an exponential tilting of the transition probability $f_h$ that penalizes based on L-divergence from the Monge map. 
Let $\Phi(p,q):= c(p,q) - \ldiv[q \mid p^*]$. 
Thus, from \eqref{eq:feepq}, 
\eq\label{eq:whatisphipq}
\Phi(p,q)=c(p, p^*) +  \frac{1}{n} \sum_{i=1}^n \log p^*_i - \varphi(p^*) - \frac{1}{n} \sum_{i=1}^n \log q_i + \varphi(q).
\en
As before, because $\rho_0$ and $\rho_1$ are compactly supported on $\Delta_n$, we will assume without loss of generality that $c(p,q)=\infty$ outside a compact subset of $\Delta_n$. Also assume that the divergence is infinity unless $q\in \support(\rho_1)$ and $p\in \support(\rho_0)$.
Hence, all integrals below are finite. 

The normalizing function
\[
Z_h(p) := \frac{1}{\nu_0(p^*)} \int_{\Delta_n} \exp\left( -\frac{1}{h} {\bf L}[q \mid p^*] \right) \nu_0(q) dq,
\]
defines an exponential tilting of the transition density $f_h$ on $\Delta_n$:
\[
\begin{split}
\tf_h(p, q) &=\frac{1}{Z_h(p) \nu_0(p^*)} \exp\left[ \frac{1}{h} \Phi(p,q) \right] f_{h}(p, q)\\
&= \frac{\Gamma(1/h)n^{-1/h}}{(\Gamma(1/nh))^n}\frac{1}{Z_h(p)} \frac{\nu_0(q)}{\nu_0(p^*)}  \exp\left[ -\frac{1}{h} c(p,q)+\frac{1}{h} \Phi(p,q)\right]\\
&= \frac{\Gamma(1/h)n^{-1/h}}{(\Gamma(1/nh))^n}\frac{1}{Z_h(p)} \frac{\nu_0(q)}{\nu_0(p^*)} \exp\left( -\frac{1}{h} \ldiv[q \mid p^*] \right).
\end{split}
\] 

Observe that, for $h\approx 0$, $\exp\left( -\frac{1}{h} \ldiv[q\mid p^*] \right)$ is very small, unless $q \approx p^*$. In the latter case, we have the Gaussian approximation from \eqref{eq:ldivquad} and \eqref{eq:stirlingfh}:
\[
\frac{\Gamma(1/h)n^{-1/h}}{(\Gamma(1/nh))^n} \exp\left[ -\frac{1}{h} {\bf L}[q \mid p^*] \right] \approx \frac{n^{-n/2}}{(2\pi h)^{(n-1)/2}}\exp\left(  -\frac{1}{2h}({q} - {p^*})^T {L}({p^*}) ({q} - {p^*}) \right).
\] 
Now, consider the $n \times (n-1)$ matrix $S$ such that $q-p= S\left( \tq - \tp\right)$. That is, 
\[
S=\begin{bmatrix}
I_{(n-1) \times (n-1)} \\
-\mathbf{1}_{n-1}
\end{bmatrix},
\]
where $I_{(n-1) \times (n-1)}$ is the $(n-1) \times (n-1)$ identity matrix and $\mathbf{1}_{n-1}$ is the row vector of all ones of length $(n-1)$. Thus 
\[
({q} - {p^*})^T {L}({p^*}) ({q} - {p^*}) = (\wt{q} - \wt{p^*})^T S^T{L}({p^*})S (\wt{q} - \wt{p^*}) =  (\wt{q} - \wt{p^*})^T \wt{L}({p^*}) (\wt{q} - \wt{p^*}), 
\]
where, obviously $\wt{L}(\cdot)= S^T L(\cdot)S$ is still nonnegative definite.

Suppose we forget that we are on the simplex, and consider the measure with the unnormalized density as above, it gives us a Gaussian distribution with mean $\wt{p^*}$ and a covariance matrix $A$ such that $A^{-1}=h \wt{L}({p^*})$. Because $\rho_0$ is compactly supported in $\Delta_n$, there is some $\delta >0$ such that $B(p^*, \delta) \subset \Delta_n$. Hence, for small values of $h$, $\nu_0(q)  \approx \nu_0(p^*)$ since $q\in B(p, \delta)$ with exponentially high probability, and we get that $\tf_h$ is approximately the Gaussian distribution
\[
q_h(p, q)  = \frac{\abs{J}^{1/2}(p^*)}{ (2\pi h)^{(n-1)/2}} \exp\left( - \frac{1}{2h}(\wt{q} - \wt{p^*})^T \wt{L}({p^*}) (\wt{q} - \wt{p^*}) \right).
\]
where $\abs{J}(p^*)$ is the determinant of $\wt{L}({p^*})$. Thus
\[
Z_h(p) \sim \frac{n^{-n/2}}{\abs{J}^{1/2}(p^*)}, \; \text{as $h\rightarrow 0+$}.
\]

The following lemma mirrors Lemma \ref{lem:partitionfnasymp}.

\begin{lemma}\label{lem:exptilting2} Under Assumptions \ref{asmp:densityc} and \ref{asmp:continuityc}, the following statements hold.  
\begin{enumerate}[(i)]
\item The normalizing constant $Z_{h}(p)$ has the following limit a.e. on $\support(\rho_0)$.
\[
\lim_{h \rightarrow 0+}  Z_{h}(p) \sqrt{\abs{J}(p^*)}=n^{-n/2}.
\]
\item Moreover,
\[
\begin{split}
\lim_{h\rightarrow 0+}   -\int \log Z_h(p) \rho_0(p) dp &= \frac{1}{2}\left( \Ent_0(\rho_1) - \Ent_0(\rho_0) \right)\\
& - \frac{n}{2} \cost(\rho_0, \rho_1) +  \int \log \nu_0(q) \rho_1(q)dq .
\end{split}
\]
\item Finally, let $\rho_1^h$ denote the law of $Y$ under the joint distribution $\tmu_h=\rho_0(p)\wt{f}_{h}(p,q)$, then, $\rho_1^h$ converges pointwise to $\rho_1$ as $h \rightarrow 0+$, a.e. on $\support(\rho_1)$.
\end{enumerate}
\end{lemma}

\begin{proof} The proofs of (i) and (iii) are almost identical to that of Lemma \ref{lem:partitionfnasymp} (ii) and (iv) and have already been outlined above. We skip the details. 

For the proof of (ii) we use \cite[Section 4.2, 4.3]{PW18}, especially the proof of Theorem 16. The rest of the argument alludes to the notation used in that reference. However, in order to use this in our case, we will need to unpack the notations.

Consider the probability measures $\rho_0^-$ and $\rho_1^-$ from the paragraph below \eqref{eqn:deterministic.transport} and consider the problem of transporting $\rho_1^-$ to $\rho_0^-$ with cost $c$. Then (see \cite{PW18}, Section 4.3, in Proof of Theorem 16 (put $t=1$, $P_0=\rho_1^-$, $P_1=\rho_0^-$, $\Ent_{\mu_0}$ to $\Ent_0$ and change the notation for the Haar measures from $\mu_0$ to $\nu_0$ and recall $r=p^{-1}$)
\eq\label{eq:diffent0}
\begin{split}
\Ent_0(\rho_0^-) &= \Ent_0(\rho_1^-) + n\log \frac{1}{n} - n\cost(\rho_1^-, \rho^-_0) \\
&-\int \log \det\left( \wt{L}({p^{-1}})\right) d\rho_1^-(p) + 2 \int \log \nu_0\left( p^{-1}\right)d\rho^-_1(p).
\end{split}
\en
 
We now revert $\rho_0^-, \rho_1^-$ to $\rho_0, \rho_1$ by replacing the variables. First, recall that $c(p,q)=c(q^{-1}, p^{-1})$. Hence $\cost(\rho_1^-, \rho_0^-)=\cost(\rho_0, \rho_1)$. This is expected since inversion is the multiplicative analog of sign change in the quadratic Wasserstein problem. Next, recall that all integrations are being performed with respect to the coordinate system $\tp=(p_1, \ldots, p_{n-1})$.  

The Jacobian of the transformation $\wt{p} \mapsto \wt{r}: \widetilde{p^{-1}}$ has been computed in \cite[eqn. (72)]{PW18}:
\eq\label{eq:PW1872}
\abs{\frac{\partial(r_1, \ldots, r_{n-1})}{\partial(p_1,\ldots, p_{n-1} )}}= \frac{r_1 \cdots r_{n}}{p_1 \cdots p_n}= \frac{\nu_0(p)}{\nu_0(r)}.
\en
Hence, by the change of variable formula 
\[
\begin{split}
\Ent_0(\rho_0^-) &= \Ent(\rho_0^-) - \int \log \nu_0(p) d\rho^-_0(p)\\
&= \Ent(\rho_0) - \int \log \abs{\frac{\partial(r_1, \ldots, r_{n-1})}{\partial(p_1,\ldots, p_{n-1} )}} d\rho_0(p) - \int \log \nu_0(p^{-1}) d\rho_0(p) \\
&= \Ent(\rho_0) - \rho_0\left( \log \nu_0(p)\right)= \Ent_0(\rho_0).
\end{split}
\]
Similarly $\Ent_0(\rho_1^-)=\Ent_0(\rho_1)$. Again, these relationships are expected since inversions with respect to $\odot$ are multiplicative analogs of sign-changes and $\nu_0$ is the multiplicative analog of the Lebesgue measure. 

Finally, by a change of variable,
\[
\begin{split}
-\int \log \det\left( \wt{L}({p^{-1}}) \right) d\rho_1^-(p) &=- \int \log \det\left( \wt{L}(q) \right) d\rho_1(q)=-\rho_0\left( \log \abs{J}(p^*)\right),\quad \text{and} \\
2\int \log \nu_0(p^{-1})d\rho_1^-(p) &= 2\rho_1(\log\nu_0(q)). 
\end{split}
\]
Hence, if we translate all the terms from \eqref{eq:diffent0}, we get 
\[
\Ent_0(\rho_0) = \Ent_0(\rho_1) + n \log\frac{1}{n} - n \cost(\rho_0, \rho_1) -\rho_0\left( \log \abs{J}(p^*)\right)+ 2\rho_1(\log\nu_0(q)).
\]
Rearranging terms gives us
\[
\rho_0\left( \log \abs{J}(p^*)\right)= \Ent_0(\rho_1) - \Ent_0(\rho_0) - n \log n - n \cost(\rho_0, \rho_1) + 2\rho_1(\log\nu_0(q)).
\]
Hence, by part (i) of this lemma, 
\[
\begin{split}
\lim_{h\rightarrow 0+} & -\int \log Z_h(p) \rho_0(p) dp= \frac{1}{2}\rho_0\left( \log \abs{J}(p^*)\right) + \frac{n}{2} \log n \\
&= \frac{1}{2}\left( \Ent_0(\rho_1) - \Ent_0(\rho_0) \right) - \frac{n}{2} \cost(\rho_0, \rho_1) + \rho_1\left(\log\nu_0(q) \right).
\end{split}
\]
This completes the proof of the lemma.
\end{proof}

\begin{proof}[Proof of Theorem \ref{thm:mainthm2}] We now complete the proof exactly as the proof of Theorem \ref{thm:mainthmw}.
Define $\tmu_h(p,q)=\rho_0(p)\tf_h(p, q)$. Then, for any $\nu \in \Pi(\rho_0, \rho_1)$, 
\[
\begin{split}
H\left(\nu \mid \mu_h \right) &= \nu\left( \log \frac{d\tmu_h}{d\mu_h}\right) + H\left( \nu \mid \tmu_h\right)\\
&= \frac{1}{h}\nu\left( \Phi(p,q)\right) - \nu\left(\log Z_h(p) \right) - \nu\left( \log \nu_0(p^*)\right) + H\left( \nu \mid \tmu_h\right)\\
&= \frac{1}{h} \cost(\rho_0, \rho_1) - \int \log Z_h(p) \rho_0(p) dp - \rho_0\left( \log \nu_0(p^*)\right) + H\left( \nu \mid \tmu_h\right). 
\end{split}
\]
The last line above comes from \eqref{eq:whatisphipq}, in particular, from the facts that
\[
\begin{split}
\nu\left( c(p,p^*)\right)&=\rho_0\left( c(p,p^*)\right)= \cost(\rho_0, \rho_1).\\
\rho_0\left( \frac{1}{n} \sum_{i=1}^n \log p_i^*\right)&= \rho_1\left( \frac{1}{n} \sum_{i=1}^n \log q_i \right), \quad
\rho_0\left(\varphi(p^*) \right) = \rho_1\left( \varphi(q)\right).
\end{split}
\]

As in proof of Theorem \ref{thm:mainthmw}, as $h\rightarrow 0+$, we get 
\[
\begin{split}
K_h(\rho_0, \rho_1 ) &= \frac{1}{h} \cost(\rho_0, \rho_1) - \int \log Z_h(p) \rho_0(p) dp - \rho_0\left( \log \nu_0(p^*)\right) + o(1) \\
&= \frac{1}{h} \cost(\rho_0, \rho_1) - \int \log Z_h(p) \rho_0(p) dp - \rho_1\left( \log \nu_0(q)\right) + o(1) \\
&= \left(\frac{1}{h} - \frac{n}{2} \right) \cost(\rho_0, \rho_1) + \frac{1}{2}\left( \Ent_0(\rho_1) - \Ent_0(\rho_0) \right) + o(1),
\end{split}
\]
from Lemma \ref{lem:exptilting2} (ii). This gives us Theorem \ref{thm:mainthm2}. 
\end{proof}

\bibliographystyle{abbrv}

\bibliography{infogeo.new}

\begin{thebibliography}{10}

\bibitem{ADPZ11}
S.~Adams, N.~Dirr, M.~A. Peletier, and J.~Zimmer.
\newblock From a large-deviations principle to the {W}asserstein gradient flow:
  a new micro-macro passage.
\newblock {\em Communications in Mathematical Physics}, 307(3):791--815, 2011.

\bibitem{A16}
S.-i. Amari.
\newblock {\em Information {G}eometry and {I}ts {A}pplications}.
\newblock Springer, 2016.

\bibitem{AGS08}
L.~Ambrosio, N.~Gigli, and G.~Savar{\'e}.
\newblock {\em Gradient flows: In Metric Spaces and in the Space of Probability
  Measures, 2nd. Edition}.
\newblock Lectures in Mathematics. Birkhauser, Basel, 2008.

\bibitem{conforti2018}
G.~Conforti.
\newblock A second order equation for {S}chr\"{o}dinger bridges with
  applications to the hot gas experiment and entropic transportation cost.
\newblock {\em Probability Theory and Related Fields}, 174(1--2):1--47, 2019.

\bibitem{Cut13}
M.~Cuturi.
\newblock Sinkhorn distance: Lightspeed computation of optimal transport.
\newblock In {\em Advances in Neural Information Processing Systems},
  volume~26, pages 2292--2300. MIT Press, Cambridge, MA, 2013.

\bibitem{dietert2015}
H.~Dietert.
\newblock Characterisation of gradient flows on finite state {M}arkov chains.
\newblock {\em Electron. Commun. Probab.}, 20(29):8 pp., 2015.

\bibitem{DLR13}
M.~H. Duong, V.~Laschos, and M.~Renger.
\newblock Wasserstein gradient flows from large deviations of many-particle
  limits.
\newblock {\em ESAIM: Control, Optimisation and Calculus of Variations},
  19(4):1166--1188, 2013.
\newblock Erratum at
  \texttt{www.wias-berlin.de/people/renger/Erratum/DLR2015ErratumFinal.pdf}.

\bibitem{Erbar14}
M.~Erbar.
\newblock Gradient flows of the entropy for jump processes.
\newblock {\em Annales de l'{I}nstitut {H}enri {P}oincar\'{e}}, 50(3):920--945,
  2014.

\bibitem{EFLS}
M.~Erbar, M.~Fathi, V.~Laschos, and A.~Schlichting.
\newblock Gradient flow structure for {M}ckean-{V}lasov equations on discrete
  spaces.
\newblock {\em Discrete {\&} Continuous Dynamical Systems - A}, 36:6799--6833,
  2016.

\bibitem{EM12}
M.~Erbar and J.~Maas.
\newblock Ricci curvature of finite {M}arkov chains via convexity of the
  entropy.
\newblock {\em Archive for Rational Mechanics and Analysis}, 206(3):997--1038,
  2012.

\bibitem{EM14}
M.~Erbar and J.~Maas.
\newblock Gradient flow structures for discrete porous medium equations.
\newblock {\em Discrete {\&} Continuous Dynamical Systems - A}, 34:1355--1374,
  2014.

\bibitem{EMR15}
M.~Erbar, J.~Maas, and D.~R.~M. Renger.
\newblock From large deviations to {W}asserstein gradient flows in multiple
  dimensions.
\newblock {\em Electronic Communications in Probability}, 20(89):1--12, 2015.

\bibitem{F02}
E.~R. Fernholz.
\newblock {\em Stochastic Portfolio Theory}.
\newblock Applications of Mathematics. Springer, 2002.

\bibitem{GM96}
W.~Gangbo and R.~J. McCann.
\newblock The geometry of optimal transportation.
\newblock {\em Acta Mathematica}, 177(2):113--161, 1996.

\bibitem{JKO98}
R.~Jordan, D.~Kinderlehrer, and F.~Otto.
\newblock The variational formulation of the {F}okker--{P}lanck equation.
\newblock {\em SIAM Journal on Mathematical Analysis}, 29(1):1--17, 1998.

\bibitem{KST}
I.~Karatzas, W.~Scahchermayer, and B.~Tschiderer.
\newblock Pathwise {O}tto calculus.
\newblock Available at math arxiv 1811.08686v2, 2019.

\bibitem{KZ19}
G.~Khan and J.~Zhang.
\newblock The {K}\"{a}hler geometry of certain optimal transport problems.
\newblock Available at math arxiv:1812.00032v3, 2019.

\bibitem{L12}
C.~L{\'e}onard.
\newblock From the {S}chr{\"o}dinger problem to the {M}onge-{K}antorovich
  problem.
\newblock {\em Journal of Functional Analysis}, 262(4):1879--1920, 2012.

\bibitem{leonardsurvey}
C.~L{\'e}onard.
\newblock A survey of the {S}chr\"{o}dinger problem and some of its connections
  with optimal transport.
\newblock {\em Discrete Contin. Dyn. Syst.}, 34(4):1533--1574, 2014.

\bibitem{MTW}
X.~N. Ma, N.~S. Trudinger, and X.~J. Wang.
\newblock Regularity of potential functions of the optimal transportation
  problem.
\newblock {\em Arch. Ration. Mech. Anal.}, 177:151--183, 2005.

\bibitem{Maas11}
J.~Maas.
\newblock Gradient flows of entropy for finite {M}arkov chains.
\newblock {\em Journal of Functional Analysis}, 261:2250--2292, 2011.

\bibitem{M97}
R.~J. McCann.
\newblock A convexity principle for interacting gases.
\newblock {\em Advances in Mathematics}, 128(1):153--179, 1997.

\bibitem{P17}
S.~Pal.
\newblock Embedding optimal transports in statistical manifolds.
\newblock {\em Indian Journal of Pure and Applied Mathematics}, 48(4):541--550,
  2017.

\bibitem{PW14}
S.~{Pal} and T.-K.~L. {Wong}.
\newblock The geometry of relative arbitrage.
\newblock {\em Mathematics and Financial Economics}, 10:263--293, 2016.

\bibitem{PW16}
S.~Pal and T.-K.~L. Wong.
\newblock Exponentially concave functions and a new information geometry.
\newblock {\em The Annals of Probability}, 46(2):1070--1113, 2018.

\bibitem{PW18}
S.~Pal and T.-K.~L. Wong.
\newblock Multiplicative {S}chr\"odinger problem and the {D}irichlet transport.
\newblock arXiv preprint 1806.05649, 2018.

\bibitem{PC19}
G.~Peyr\'{e} and M.~Cuturi.
\newblock Computational optimal transport.
\newblock {\em Foundations and Trends in Machine Learning}, 11(5-6):355--607,
  2019.

\bibitem{W15}
T.-K.~L. Wong.
\newblock Optimization of relative arbitrage.
\newblock {\em Annals of Finance}, 11(3-4):345--382, 2015.

\bibitem{W17}
T.-K.~L. Wong.
\newblock Logarithmic divergences from optimal transport and {R}\'{e}nyi
  geometry.
\newblock {\em Information Geometry}, 1(1):39--78, 2018.

\end{thebibliography}

\end{document}